\documentclass[10pt,leqno]{amsart}
\usepackage{amssymb}
\usepackage{xypic}
\begin{document}
\theoremstyle{plain}
\newtheorem{thm}{Theorem}[section]
\newtheorem*{thm1}{Theorem 1}
\newtheorem*{thm2}{Theorem 2}
\newtheorem{lemma}[thm]{Lemma}
\newtheorem{lem}[thm]{Lemma}
\newtheorem{cor}[thm]{Corollary}
\newtheorem{prop}[thm]{Proposition}
\newtheorem{propose}[thm]{Proposition}
\newtheorem{variant}[thm]{Variant}
\theoremstyle{definition}
\newtheorem{notations}[thm]{Notations}
\newtheorem{rem}[thm]{Remark}
\newtheorem{rmk}[thm]{Remark}
\newtheorem{rmks}[thm]{Remarks}
\newtheorem{defn}[thm]{Definition}
\newtheorem{ex}[thm]{Example}
\newtheorem{claim}[thm]{Claim}
\newtheorem{ass}[thm]{Assumption}
\numberwithin{equation}{section}
\newcounter{elno}                
\def\points{\list
{\hss\llap{\upshape{(\roman{elno})}}}{\usecounter{elno}}} 
\let\endpoints=\endlist


\catcode`\@=11
%
%
\def\opn#1#2{\def#1{\mathop{\kern0pt\fam0#2}\nolimits}} 
\def\bold#1{{\bf #1}}%
\def\underrightarrow{\mathpalette\underrightarrow@}
\def\underrightarrow@#1#2{\vtop{\ialign{$##$\cr
 \hfil#1#2\hfil\cr\noalign{\nointerlineskip}%
 #1{-}\mkern-6mu\cleaders\hbox{$#1\mkern-2mu{-}\mkern-2mu$}\hfill
 \mkern-6mu{\to}\cr}}}
\let\underarrow\underrightarrow
\def\underleftarrow{\mathpalette\underleftarrow@}
\def\underleftarrow@#1#2{\vtop{\ialign{$##$\cr
 \hfil#1#2\hfil\cr\noalign{\nointerlineskip}#1{\leftarrow}\mkern-6mu
 \cleaders\hbox{$#1\mkern-2mu{-}\mkern-2mu$}\hfill
 \mkern-6mu{-}\cr}}}
%
%

%
\def\:{\colon}
\let\oldtilde=\tilde
\def\tilde#1{\mathchoice{\widetilde{#1}}{\widetilde{#1}}%
{\indextil{#1}}{\oldtilde{#1}}}
\def\indextil#1{\lower2pt\hbox{$\textstyle{\oldtilde{\raise2pt%
\hbox{$\scriptstyle{#1}$}}}$}}
\def\pnt{{\raise1.1pt\hbox{$\textstyle.$}}}
%

%
\let\amp@rs@nd@\relax
\newdimen\ex@\ex@.2326ex
\newdimen\bigaw@
\newdimen\minaw@
\minaw@16.08739\ex@
\newdimen\minCDaw@
\minCDaw@2.5pc
\newif\ifCD@
\def\minCDarrowwidth#1{\minCDaw@#1}
\newenvironment{CD}{\@CD}{\@endCD}
\def\@CD{\def\A##1A##2A{\llap{$\vcenter{\hbox
 {$\scriptstyle##1$}}$}\Big\uparrow\rlap{$\vcenter{\hbox{%
$\scriptstyle##2$}}$}&&}%
\def\V##1V##2V{\llap{$\vcenter{\hbox
 {$\scriptstyle##1$}}$}\Big\downarrow\rlap{$\vcenter{\hbox{%
$\scriptstyle##2$}}$}&&}%
\def\={&\hskip.5em\mathrel
 {\vbox{\hrule width\minCDaw@\vskip3\ex@\hrule width
 \minCDaw@}}\hskip.5em&}%
\def\verteq{\Big\Vert&&}%
\def\noarr{&&}%
\def\vspace##1{\noalign{\vskip##1\relax}}\relax\let\amp@rs@nd@&\iffalse}\fi
 \CD@true\vcenter\bgroup\relax\let\\=\cr\iffalse}\fi\tabskip\z@skip\baselineskip20\ex@
 \lineskip3\ex@\lineskiplimit3\ex@\halign\bgroup
 &\hfill$\m@th##$\hfill\cr}
\def\@endCD{\cr\egroup\egroup}
%
\def\>#1>#2>{\amp@rs@nd@\setbox\z@\hbox{$\scriptstyle
 \;{#1}\;\;$}\setbox\@ne\hbox{$\scriptstyle\;{#2}\;\;$}\setbox\tw@
 \hbox{$#2$}\ifCD@
 \global\bigaw@\minCDaw@\else\global\bigaw@\minaw@\fi
 \ifdim\wd\z@>\bigaw@\global\bigaw@\wd\z@\fi
 \ifdim\wd\@ne>\bigaw@\global\bigaw@\wd\@ne\fi
 \ifCD@\hskip.5em\fi
 \ifdim\wd\tw@>\z@
 \mathrel{\mathop{\hbox to\bigaw@{\rightarrowfill}}\limits^{#1}_{#2}}\else
 \mathrel{\mathop{\hbox to\bigaw@{\rightarrowfill}}\limits^{#1}}\fi
 \ifCD@\hskip.5em\fi\amp@rs@nd@}
\def\<#1<#2<{\amp@rs@nd@\setbox\z@\hbox{$\scriptstyle
 \;\;{#1}\;$}\setbox\@ne\hbox{$\scriptstyle\;\;{#2}\;$}\setbox\tw@
 \hbox{$#2$}\ifCD@
 \global\bigaw@\minCDaw@\else\global\bigaw@\minaw@\fi
 \ifdim\wd\z@>\bigaw@\global\bigaw@\wd\z@\fi
 \ifdim\wd\@ne>\bigaw@\global\bigaw@\wd\@ne\fi
 \ifCD@\hskip.5em\fi
 \ifdim\wd\tw@>\z@
 \mathrel{\mathop{\hbox to\bigaw@{\leftarrowfill}}\limits^{#1}_{#2}}\else
 \mathrel{\mathop{\hbox to\bigaw@{\leftarrowfill}}\limits^{#1}}\fi
 \ifCD@\hskip.5em\fi\amp@rs@nd@}
%
%
\newenvironment{CDS}{\@CDS}{\@endCDS}
\def\@CDS{\def\A##1A##2A{\llap{$\vcenter{\hbox
 {$\scriptstyle##1$}}$}\Big\uparrow\rlap{$\vcenter{\hbox{%
$\scriptstyle##2$}}$}&}%
\def\V##1V##2V{\llap{$\vcenter{\hbox
 {$\scriptstyle##1$}}$}\Big\downarrow\rlap{$\vcenter{\hbox{%
$\scriptstyle##2$}}$}&}%
\def\={&\hskip.5em\mathrel
 {\vbox{\hrule width\minCDaw@\vskip3\ex@\hrule width
 \minCDaw@}}\hskip.5em&}
\def\verteq{\Big\Vert&}
\def\novarr{&}
\def\noharr{&&}
\def\SE##1E##2E{\slantedarrow(0,18)(4,-3){##1}{##2}&}
\def\SW##1W##2W{\slantedarrow(24,18)(-4,-3){##1}{##2}&}
\def\NE##1E##2E{\slantedarrow(0,0)(4,3){##1}{##2}&}
\def\NW##1W##2W{\slantedarrow(24,0)(-4,3){##1}{##2}&}
\def\slantedarrow(##1)(##2)##3##4{%
\thinlines\unitlength1pt\lower 6.5pt\hbox{\begin{picture}(24,18)%
\put(##1){\vector(##2){24}}%
\put(0,8){$\scriptstyle##3$}%
\put(20,8){$\scriptstyle##4$}%
\end{picture}}}
\def\vspace##1{\noalign{\vskip##1\relax}}\relax\let\amp@rs@nd@&\iffalse}\fi
 \CD@true\vcenter\bgroup\relax\let\\=\cr\iffalse}\fi\tabskip\z@skip\baselineskip20\ex@
 \lineskip3\ex@\lineskiplimit3\ex@\halign\bgroup
 &\hfill$\m@th##$\hfill\cr}
\def\@endCDS{\cr\egroup\egroup}
%
\newdimen\TriCDarrw@
\newif\ifTriV@
\newenvironment{TriCDV}{\@TriCDV}{\@endTriCD}
\newenvironment{TriCDA}{\@TriCDA}{\@endTriCD}
\def\@TriCDV{\TriV@true\def\TriCDpos@{6}\@TriCD}
\def\@TriCDA{\TriV@false\def\TriCDpos@{10}\@TriCD}
\def\@TriCD#1#2#3#4#5#6{%
\setbox0\hbox{$\ifTriV@#6\else#1\fi$}
\TriCDarrw@=\wd0 \advance\TriCDarrw@ 24pt
\advance\TriCDarrw@ -1em
\def\SE##1E##2E{\slantedarrow(0,18)(2,-3){##1}{##2}&}
\def\SW##1W##2W{\slantedarrow(12,18)(-2,-3){##1}{##2}&}
\def\NE##1E##2E{\slantedarrow(0,0)(2,3){##1}{##2}&}
\def\NW##1W##2W{\slantedarrow(12,0)(-2,3){##1}{##2}&}
\def\slantedarrow(##1)(##2)##3##4{\thinlines\unitlength1pt
\lower 6.5pt\hbox{\begin{picture}(12,18)%
\put(##1){\vector(##2){12}}%
\put(-4,\TriCDpos@){$\scriptstyle##3$}%
\put(12,\TriCDpos@){$\scriptstyle##4$}%
\end{picture}}}
\def\={\mathrel {\vbox{\hrule
   width\TriCDarrw@\vskip3\ex@\hrule width
   \TriCDarrw@}}}
\def\>##1>>{\setbox\z@\hbox{$\scriptstyle
 \;{##1}\;\;$}\global\bigaw@\TriCDarrw@
 \ifdim\wd\z@>\bigaw@\global\bigaw@\wd\z@\fi
 \hskip.5em
 \mathrel{\mathop{\hbox to \TriCDarrw@
{\rightarrowfill}}\limits^{##1}}
 \hskip.5em}
\def\<##1<<{\setbox\z@\hbox{$\scriptstyle
 \;{##1}\;\;$}\global\bigaw@\TriCDarrw@
 \ifdim\wd\z@>\bigaw@\global\bigaw@\wd\z@\fi
 \mathrel{\mathop{\hbox to\bigaw@{\leftarrowfill}}\limits^{##1}}
 }
 \CD@true\vcenter\bgroup\relax\let\\=\cr\iffalse}\fi
 \tabskip\z@skip\baselineskip20\ex@
 \lineskip3\ex@\lineskiplimit3\ex@
 \ifTriV@
 \halign\bgroup
 &\hfill$\m@th##$\hfill\cr
#1&\multispan3\hfill$#2$\hfill&#3\\
&#4&#5\\
&&#6\cr\egroup%
\else
 \halign\bgroup
 &\hfill$\m@th##$\hfill\cr
&&#1\\%
&#2&#3\\
#4&\multispan3\hfill$#5$\hfill&#6\cr\egroup
\fi}
\def\@endTriCD{\egroup} 
\newcommand{\mc}{\mathcal} 
\newcommand{\mb}{\mathbb} 
\newcommand{\surj}{\twoheadrightarrow} 
\newcommand{\inj}{\hookrightarrow} \newcommand{\zar}{{\rm zar}} 
\newcommand{\an}{{\rm an}} \newcommand{\red}{{\rm red}} 
\newcommand{\Rank}{{\rm rk}} \newcommand{\codim}{{\rm codim}} 
\newcommand{\rank}{{\rm rank}} \newcommand{\Ker}{{\rm Ker \ }} 
\newcommand{\Pic}{{\rm Pic}} \newcommand{\Div}{{\rm Div}} 
\newcommand{\Hom}{{\rm Hom}} \newcommand{\im}{{\rm im}} 
\newcommand{\Spec}{{\rm Spec \,}} \newcommand{\Sing}{{\rm Sing}} 
\newcommand{\sing}{{\rm sing}} \newcommand{\reg}{{\rm reg}} 
\newcommand{\Char}{{\rm char}} \newcommand{\Tr}{{\rm Tr}} 
\newcommand{\Gal}{{\rm Gal}} \newcommand{\Min}{{\rm Min \ }} 
\newcommand{\Max}{{\rm Max \ }} \newcommand{\Alb}{{\rm Alb}\,} 
\newcommand{\GL}{{\rm GL}\,} 
\newcommand{\ie}{{\it i.e.\/},\ } \newcommand{\niso}{\not\cong} 
\newcommand{\nin}{\not\in} 
\newcommand{\soplus}[1]{\stackrel{#1}{\oplus}} 
\newcommand{\by}[1]{\stackrel{#1}{\rightarrow}} 
\newcommand{\longby}[1]{\stackrel{#1}{\longrightarrow}} 
\newcommand{\vlongby}[1]{\stackrel{#1}{\mbox{\large{$\longrightarrow$}}}} 
\newcommand{\ldownarrow}{\mbox{\Large{\Large{$\downarrow$}}}} 
\newcommand{\lsearrow}{\mbox{\Large{$\searrow$}}} 
\renewcommand{\d}{\stackrel{\mbox{\scriptsize{$\bullet$}}}{}} 
\newcommand{\dlog}{{\rm dlog}\,} 
\newcommand{\longto}{\longrightarrow} 
\newcommand{\vlongto}{\mbox{{\Large{$\longto$}}}} 
\newcommand{\limdir}[1]{{\displaystyle{\mathop{\rm lim}_{\buildrel\longrightarrow\over{#1}}}}\,} 
\newcommand{\liminv}[1]{{\displaystyle{\mathop{\rm lim}_{\buildrel\longleftarrow\over{#1}}}}\,} 
\newcommand{\norm}[1]{\mbox{$\parallel{#1}\parallel$}} 
\newcommand{\boxtensor}{{\Box\kern-9.03pt\raise1.42pt\hbox{$\times$}}} 
\newcommand{\into}{\hookrightarrow} \newcommand{\image}{{\rm image}\,} 
\newcommand{\Lie}{{\rm Lie}\,} 
\newcommand{\CM}{\rm CM}
\newcommand{\sext}{\mbox{${\mathcal E}xt\,$}} 
\newcommand{\shom}{\mbox{${\mathcal H}om\,$}} 
\newcommand{\coker}{{\rm coker}\,} 
\newcommand{\sm}{{\rm sm}} 
\newcommand{\tensor}{\otimes} 
\renewcommand{\iff}{\mbox{ $\Longleftrightarrow$ }} 
\newcommand{\supp}{{\rm supp}\,} 
\newcommand{\ext}[1]{\stackrel{#1}{\wedge}} 
\newcommand{\onto}{\mbox{$\,\>>>\hspace{-.5cm}\to\hspace{.15cm}$}} 
\newcommand{\propsubset} {\mbox{$\textstyle{ 
\subseteq_{\kern-5pt\raise-1pt\hbox{\mbox{\tiny{$/$}}}}}$}} 
\newcommand{\sB}{{\mathcal B}} \newcommand{\sC}{{\mathcal C}} 
\newcommand{\sD}{{\mathcal D}} \newcommand{\sE}{{\mathcal E}} 
\newcommand{\sF}{{\mathcal F}} \newcommand{\sG}{{\mathcal G}} 
\newcommand{\sH}{{\mathcal H}} \newcommand{\sI}{{\mathcal I}} 
\newcommand{\sJ}{{\mathcal J}} \newcommand{\sK}{{\mathcal K}} 
\newcommand{\sL}{{\mathcal L}} \newcommand{\sM}{{\mathcal M}} 
\newcommand{\sN}{{\mathcal N}} \newcommand{\sO}{{\mathcal O}} 
\newcommand{\sP}{{\mathcal P}} \newcommand{\sQ}{{\mathcal Q}} 
\newcommand{\sR}{{\mathcal R}} \newcommand{\sS}{{\mathcal S}} 
\newcommand{\sT}{{\mathcal T}} \newcommand{\sU}{{\mathcal U}} 
\newcommand{\sV}{{\mathcal V}} \newcommand{\sW}{{\mathcal W}} 
\newcommand{\sX}{{\mathcal X}} \newcommand{\sY}{{\mathcal Y}} 
\newcommand{\sZ}{{\mathcal Z}} \newcommand{\ccL}{\sL} 
 \newcommand{\A}{{\mathbb A}} \newcommand{\B}{{\mathbb 
B}} \newcommand{\C}{{\mathbb C}} \newcommand{\D}{{\mathbb D}} 
\newcommand{\E}{{\mathbb E}} \newcommand{\F}{{\mathbb F}} 
\newcommand{\G}{{\mathbb G}} \newcommand{\HH}{{\mathbb H}} 
\newcommand{\I}{{\mathbb I}} \newcommand{\J}{{\mathbb J}} 
\newcommand{\M}{{\mathbb M}} \newcommand{\N}{{\mathbb N}} 
\renewcommand{\P}{{\mathbb P}} \newcommand{\Q}{{\mathbb Q}} 

\newcommand{\R}{{\mathbb R}} \newcommand{\T}{{\mathbb T}} 
\newcommand{\U}{{\mathbb U}} \newcommand{\V}{{\mathbb V}} 
\newcommand{\W}{{\mathbb W}} \newcommand{\X}{{\mathbb X}} 
\newcommand{\Y}{{\mathbb Y}} \newcommand{\Z}{{\mathbb Z}} 
\title[Towards Hilbert-Kunz density functions in Characteristic $0$] 
{Towards Hilbert-Kunz density functions in Characteristic $0$} 
\author{V. Trivedi} \address{School of Mathematics, Tata Institute of 
Fundamental Research, Homi Bhabha Road, Mumbai-400005, India} 
\email{vija@math.tifr.res.in} \date{}
\begin{abstract}For a pair $(R, I)$, where $R$ is a standard graded
 domain of dimension $d$
over an algebraically closed  field of characteristic $0$ and $I$ 
is a graded ideal of finite colength, we prove 
that the existence of $\lim_{p\to \infty}e_{HK}(R_p, I_p)$
 is equivalent, for any fixed $m\geq d-1$,  to the existence of
 $\lim_{p\to \infty}\ell(R_p/I_p^{[p^m]})/p^{md}$. 
 
This we get as a consequence of
Theorem~1.1: As
$p\longto \infty $,  
the convergence of the HK density function 
$f{(R_p, I_p)}$  is equivalent to the convergence of the truncated HK density 
functions $f_m(R_p, I_p)$ (in $L^{\infty}$ norm) of the {\it mod $p$ reductions}
$(R_p, I_p)$, for any fixed $m\geq d-1$.

In particular,  to define the HK density function $f^{\infty}(R, I)$ in $\Char~0$, 
 it is enough to 
 prove the existence of $\lim_{p\to \infty} f_m(R_p, I_p)$,
for any fixed $m\geq d-1$. 

This allows us to prove the existence of $e_{HK}^{\infty}(R, I)$ in many new cases, 
{\em e.g.},  
when $\mbox{Proj~R}$ is a Segre product of curves, for example.
\end{abstract}


\maketitle\section{Introduction}

Let $R$ be a  Noetherian ring of prime characteristic
$p >0$ and of dimension $d$ and let
$I\subseteq R$ be an ideal of finite colength. Then we recall that
the {\em Hilbert-Kunz
multiplicity} of $R$ with
respect to $I$ is defined as
$$e_{HK}(R, I) = \lim_{n\to \infty}\frac{\ell(R/I^{[p^n]})}{p^{nd}},$$
where
$ I^{[p^n]} =$ the  $n$-th Frobenius power of $I$
= the ideal generated by $p^n$-th power of elements of $I$. This   
 is an ideal of finite colength and  $\ell(R/I^{[p^n]})$ denotes the
length of
the
$R$-module $R/I^{[p^n]}$.  
This invariant had been introduced by E. Kunz and
existence of the limit was proved by 
Monsky [Mo1]. It  carries information about  char~$p$ 
related properties of the ring,  but  at the same time is difficult to
compute (even in the graded case)
 as various standard techniques, used for studying multiplicities, 
are not applicable for the invariant $e_{HK}$.

 It is natural to ask if the notion of this invariant can be extended to
the  `char~0' case by studying the behaviour of {\it mod~$p$ reductions}. 

A natural way to attempt this for {\em a pair $(R,I)$} (from now onwards, 
unless stated otherwise, 
 by a pair  $(R, I)$, we mean  $R$ is a  
standard graded ring and $I\subset R$ is a graded  ideal of finite colength)
could  be as follows: Suppose $R$ is a finitely generated
 algebra  and a domain over a field  $k$ of characteristic $0$ and 
$I\subseteq R$ is an ideal
 of finite colength. Let $(A, R_A, I_A)$ be a {\em spread} of the pair $(R, I)$
 (see Definition~\ref{d3}),
 where $A \subset k$ is a finitely generated algebra over $\Z$.
Then we may define
$$e_{HK}^{\infty}(R, I) := \lim_{s\to s_0}e_{HK}(R_s, I_s),$$
 where $R_s = R_A\tensor_A{\bar k(s)}$ and 
$I_s = I_A\tensor_A {\bar k(s)}$ with ${\bar k(s)}$ as the algebraic closure
 of $k(s)$,    
$s_0$ is the generic point of $\Spec(A)$, and $s$ is a closed point of 
$\Spec(A)$ (the 
definition is tentative, since the existence of this limit is not known in general).
Or consider a simpler situation: $R$ is a finitely generated $\Z$-algebra and a
 domain, $I\subset R$ such that $R/I $ is an abelian group of finite rank
then let 
$$e_{HK}^{\infty}(R, I) := \lim_{p\to \infty}e_{HK}(R_p, I_p),
\quad\mbox{where}\quad R_p= R\tensor_{\Z}\Z/p\Z\quad\mbox{and}\quad 
I_p= I\tensor_{\Z}\Z/p\Z.$$

In case of dimension~$R=1$, we know that the Hilbert-Kunz multiplicity
 coincides with the Hilbert-Samuel multiplicity; hence it is independent of $p$, for large $p$.

For homogeneous coordinate rings of  plane curves, with respect to the
 maximal graded ideal (in [T1], [Mo3]), 
nonsingular curves with respect to a graded ideal $I$ (in [T2]),
 diagonal 
hypersurfaces (in [GM] and [HM]), it has been shown that $e_{HK}(R_p, I_p)$ 
varies with  $p$, and  the  limit exists as $p\to \infty $. 
Then there are other cases where $e_{HK}(R_p, I_p)$ is independent of $p$:
 plane cubics (by [BC], [Mo2] and [P]), certain 
monomials ideals (by [Br], [C], [E], [W]), two dimensional invariant rings for 
finite group actions (by [WY2])
and full flag varietes and elliptic curves (by [FT]).
 Therefore the  limit exists in all these cases. 

Since $$e_{HK}^{\infty}(R, I) := \lim_{p\to \infty} 
\lim_{n\to \infty}\frac{\ell(R_p/I_p^{[p^n]})}{(p^n)^d},$$
it seems harder to compute as such, as the inner limit
$\lim_{n\to \infty}\ell(R_p/I_p^{[p^n]})/{(p^n)^d}$ itself does
 not seem easily computable (even in the  graded case). In the special 
situation considered in [GM] by 
Gessel-Monsky, the existence of $e_{HK}^{\infty}$ is proved by 
reducing the problem to the existence of 
 $\lim_{p\to \infty}
\frac{\ell(R_p/I_p^{[p]})}{p^{d}}$. 
To make this invariant  more
approachable in a general graded case,  the following question was posed in [BLM]:

\vspace{5pt}

\noindent{\bf Question}. Suppose $e_{HK}^{\infty}(R, I)$ exists, is it true 
that for any  fixed $n\geq 1$
$$ e_{HK}^{\infty}(R, I) = \lim_{p\to \infty}
\frac{\ell(R_p/I_p^{[p^n]})}{(p^n)^d}?$$

The main result of their paper was to an give affirmative answer in the case
 of a $2$ dimensional standard graded normal domain $R$ with respect to 
a homogeneous ideal $I$ of finite colength.
 Note that the existence of $e_{HK}^{\infty}(R, I)$, in this case,
 was  proved earlier in [T2]. 

Recall  that for a 
vector bundle $V$ on a smooth (projective and polarized) variety, we have the well defined {\it HN data},
namely $\{r_i(V), \mu_i(V)\}_i$, where $r_i(V) = \rank(E_i/E_{i-1})$ and 
$\mu_i(V) = \mbox{slope of}~E_i/E_{i-1}$ and 
$$0\subset E_1 \subset E_2 \subset \cdots \subset E_l \subset V$$
 is the Harder-Narasimhan filtration of $V$.
 
Let $X_p = {\rm Proj}~R_p$, which is  a nonsingular projective curve,
 and let $I_p$ be generated by homogeneous 
elements of degrees
$d_1, \ldots, d_{\mu}$ then we have  
 the vector bundle $V_p$ on $X_p$ given by the
 following  canonical exact sequence of
 $\sO_{X_p}$-modules
$$0\longto V_p\longto \oplus_i\sO_{X_p}(1-d_i)\longto \sO_{X_p}(1)\longto 0.$$
  
Then, by Proposition~1.16 in [T2], there is a constant $C$ determined by genus of 
$X_p$ and $\rank~V_p$ (hence independent of $p$), such that for $s\geq 1$
\begin{equation}\label{e1}\left|\sum_jr_j(F^{s*}V_p)\mu_j(F^{s*}V_p)^2 -
\sum_i r_i(V_p)\mu_i(V_p)^2\right| \leq C/p.\end{equation}
(Here $F$ is the absolute Frobenius morphism, and $F^s$ is the $s$-fold iterate.)
Note that the HN filtration and hence the HN data of $V_p$ stabilizes
for $p>>0$ ([Mar]).

Thus here 
\begin{enumerate}
\item one relates $\ell(R_p/I_p^{[p^s]})$ with the HN data of $F^{s*}V_p$, 
for $s\geq 1$ ([B] and [T1]),
\item the HN data of $F^{s*}V_p$ is related to the HN data of $V_p$ ([T2]),
\item the restriction of the relative HN filtration of $V_A$ on $X_A$ 
 (where $V_A$ is a spread of 
$V_0$ in char~$0$) remains the HN filtration 
of $V_p$ for large $p$ ([Mar]).
\end{enumerate}

In particular for a pair $(R,I)$, where $\Char~R = p >0$,
 with the associated syzygy bundle $V$ (as above),
 the proof uses the  comparison of  $\ell(R/I^{[p^s]})$ with the  
HN data of the syzygy bundle $V$ and the other well behaved invariants of 
($R$, $I$) (which have well defined notion in all characteristics and 
are  well behaved vis-a-vis reduction mod $p$). 
 
However  note that  $(3)$ is valid for   $\dim~R\geq 2$,  and 
 $(2)$ also holds  for 
$\dim~R\geq 3$  (proved relatively recently in [T3]).
 But $(1)$
 does not seem to hold in higher dimension, 
 due to the existence of cohomologies other than 
 $H^0(-)$ and $H^1(-)$ (therefore one can not use anymore the semistability 
property of a vector bundle to compute $h^0$ of almost all its twists,  
by powers of a very ample line bundle).

In this paper, we approach the problem by a completely different method (see 
Corollary~\ref{r4}), 
 comparing directly
$\frac{1}{(p^n)^d}\ell(R/I^{[p^n]})$ and $\frac{1}{(p^{n+1})^d}
\ell(R/I^{[p^{n+1}]})$, for $n\geq 1$, taking into account
 that both are graded.

For this we phrase the problem in a more general setting:
By  the theory of Hilbert-Kunz {\it density function} (which was 
introduced and developed in 
[T4]), for  a pair
 $(R, I)$ where $R$ is a domain  of $\Char~p>0$,
 there exists  a sequence of functions $\{f_n(R_p, I_p):[0, \infty)\longto
 \R\}_n$ such 
that  
$$\frac{1}{(p^n)^d}\ell\left(\frac{R}{I^{[p^n]}}\right) = 
\int_0^{\infty} f_n(R_p, I_p)((x)dx\quad
\mbox{and}\quad 
\lim_{n\to \infty}\frac{1}{(p^n)^d}\ell\left(\frac{R}{I^{[p^n]}}\right) = 
\int_0^{\infty} f(R_p, I_p)(x)dx,$$
where the map 
$$f(R_p, I_p):[0, \infty)\to \R\quad\mbox{ is given by}\quad
f(R_p, I_p)(x) = \lim_{n\to \infty} f_n(R_p, I_p)(x)$$
 is called the HK density function of 
$(R_p, I_p)$ (the existence and properties of the limit defining $f(R_p, I_p)$  
are proved in [T4]).
 We show here that, for each $x\in [1, \infty)$, 
$$\quad f^{\infty}(R, I)(x) := 
\lim_{p\to \infty}\lim_{n\to \infty} f_n(R_p, I_p)(x)\quad\mbox{exists}~~\iff~~ 
 \lim_{p\to \infty}f_m(R_p, I_p)(x)\quad\mbox{exists}, $$
for any fixed $m \geq d-1$, where  $d-1 = \dim{\rm Proj}~R$. Moreover
 if it exists then 
$$\quad f^{\infty}(R, I)(x) = 
 \lim_{p\to \infty}f_m(R_p, I_p)(x),\quad\mbox{for any}\quad m\geq d-1.$$
 
The main point  (Proposition~\ref{p1}) is to give a 
bound on the difference $\|f_n(R_p, I_p) - f_{n+1}(R_p, I_p)\|$,
in terms of a power of $p$ and invariants which are well behaved under reduction 
mod $p$,  where 
$\|g\| := \sup\{g(x)\mid x\in [1, \infty)\}$ is the $L^{\infty}$ norm. 
Since the union of the support of all $f_n$ is contained in a compact interval,
a similar bound   (Corollary~\ref{r4}) holds for 
the difference  $|\ell(R/I^{[p^n]})/{(p^n)^d} - 
\ell(R/I^{[p^{n+1}]})/{(p^{n+1})^d}|$.
More precisely we prove the following

\begin{thm}\label{t2}Let $R$ be a standard graded 
domain of dimension $d\geq 2$, over an algebraically closed field  $k$ of 
characteristic $0$. Let $I \subset R$ be a homogeneous ideal of finite
colength. Let $(A, R_A, I_A)$ be a spread (see Definition~\ref{d3} and
 Notations~\ref{n6}).
 Then, for 
a closed point $s\in \Spec(A)$, let the function 
$$f_n(R_s, I_s)(x):[1, \infty)\longto [0, \infty)\quad\mbox{be given by}\quad f_n(R_s, I_s)(x)= 
\frac{1}{q^{d-1}}\ell\left(\frac{R_s}{I_s^{[q]}}\right)_{\lfloor xq\rfloor}.$$ 
Let the HK density function of $(R_s, I_s)$ {be given by} 
$$ f(R_s, I_s)(x) = \lim_{n\to \infty}f_n(R_s, I_s)(x).$$
Let  $s_0 \in \Spec{Q(A)}$
 denote the generic point of $\Spec(A)$. Then 
\begin{enumerate}
\item  there exists a constant $C$ (given in terms of invariants 
of $(R_{s_0}, I_{s_0})$ of the generic fiber) and an open dense subset $\Spec(A')$ of $\Spec(A)$ 
 such that for every closed point
 $s\in \Spec(A')$ and $n\geq 1$,  
$$\|f_n(R_s, I_s)-f_{n+1}(R_s, I_s)\|<C/p^{n-d+2},$$
where  $p = \Char~k(s)$. In particular, for any  $m\geq d-1$, 
$$\lim_{s\to s_0}\|f_{m}(R_s, I_s)-f(R_s, I_s)\| = 0.$$
\item There exists a constant $C_1$ (given in terms of invariants 
of $(R_{s_0}, I_{s_0})$) and an open dense subset $\Spec(A')$ of $\Spec(A)$,
 such that for every closed point
 $s\in \Spec(A')$ and $n\geq 1$,  
 we have 
$$\left|\frac{1}{p^{nd}}\ell\left(\frac{R_s}{I_s^{[p^n]}}\right) - 
\frac{1}{p^{(n+1)d}}\ell\left(\frac{R_s}{I_s^{[p^{n+1}]}}\right)\right|\leq 
\frac{C_1}{p^{n-d+2}}.$$
\item For any $m\geq d-1$, $$\lim_{s\to s_0}\left[\frac{1}{p^{md}}
\ell\left(\frac{R_s}{I_s^{[p^{m}]}}\right)-
e_{HK}(R_s, I_s)\right] = 0.$$ 
\end{enumerate}
\end{thm}

As a result we have

\begin{cor}\label{c1}Let $R$ be a standard graded domain and a finitely 
generated $\Z$-algebra of characteristic $0$, 
let $I\subset R$ be a homogeneous ideal of finite
 colength, such that for almost all $p$, the fiber over $p$,  
$R_p:= R\tensor_{\Z}\Z/p\Z$ 
is a standard graded ring of dimension $d$, which is geometrically integral,
and $I_p\subset R_p$ is a 
homogenous ideal of finite colength. Then
\begin{enumerate}
\item there exists a constant $C_1$ given in terms of invariants of 
$R$ and $I$ such that, 
for $n\geq 1$,  
 we have 
$$\left|\frac{1}{p^{nd}}\ell\left(\frac{R_p}{I_p^{[p^n]}}\right) - 
\frac{1}{p^{(n+1)d}}\ell\left(\frac{R_p}{I_p^{[p^{n+1}]}}\right)\right|\leq 
\frac{C_1}{p^{n-d+2}}.$$
\item For any fixed  $m\geq d-1$, 
$$\lim_{p\to \infty}\left[e_{HK}(R_p, I_p)-
\frac{1}{p^{md}}\ell\left(\frac{R_p}{I_p^{[p^{m}]}}\right)\right] = 0.$$
 In particular,  for any  fixed $m\geq d-1$, 
$$e_{HK}^{\infty}(R, I):= \lim_{p\to \infty}e_{HK}(R_p, I_p)\quad\mbox{exists}~~
 \iff~~~
\lim_{p\to \infty}\frac{1}{p^{md}}\ell\left(\frac{R_p}{I_p^{[p^{m}]}}\right)\quad\mbox{exists}.$$
\end{enumerate}

\end{cor}

In particular the last assertion of the above corollary 
answers the above mentioned
 question of [BLM] affirmatively, for all $(R,I)$, where $R$ is a
standard graded domain 
and $I\subset R$  is a graded ideal of finite 
colength.

 Moreover the proof, even in the case of dimension $2$ 
(unlike the proof in [BLM])
 does not rely on earlier results of 
[B], [T1] and [T2]. In particular, since we do not use Harder-Narasimhan 
filtrations, we do not need a normality hypothesis  on the ring $R$.

\begin{rmk}\label{r2}If $e_{HK}^\infty(R,I)$ exists for a pair $(R, I)$,
whenever $R$ is a standard graded 
domain, defined over an algebraically closed field of characteristic $0$, 
then one can check that $e_{HK}^\infty(R,I)$ exists for 
any  pair $(R, I)$ where $R$ is a standard graded ring over a field
$k$  of characteristic $0$: Let ${\bar R} = R\tensor_k{\bar k}$.
Let $\{ q_1, \ldots,  q_r\} = 
\{ q\in {\rm Ass}({\bar R})\mid \dim~{\bar R}/{ q} = \dim~R\}$
then we have a spread $(A, {\bar R}_A, {\bar I}_A)$ of $({\bar R}, {\bar I})$
  such that 
 $\{{ q}_{1s}, \ldots, { q}_{rs}\} = 
\{{ q}_s\in {\rm Ass}({\bar R}_s)\mid \dim~{\bar R}_s/{ q}_s =
 \dim~{\bar R}_s\}$ (here $q_{is} = q_i\tensor_k{\bar{k(s)}} \subset {\bar R}$).
and, for each $i$, $\ell(({\bar R}_s){q_{is}})= l_i$, a constant
independent of $s$. This implies that 
$$e_{HK}({\bar R}_s,{\bar I}_s) = \sum_{i=1}^rl_ie_{HK}
\left(\frac{{\bar R}_s}{{ q}_{is}}, 
\frac{{\bar I}_s+{ q}_{is}}{{ q}_{is}}\right),$$
which implies 
$$\lim_{s\to s_0}e_{HK}({\bar R}_s,{\bar I}_s) = 
\sum_{i=1}^rl_i\lim_{s\to s_0}e_{HK}\left(\frac{{\bar R}_s}{{ q}_{is}}, 
\frac{{\bar I}_s+{ q}_{is})}{{ q}_{is}}\right) = \sum_{i=1}^rl_i
e_{HK}^{\infty}\left(\frac{{\bar R}}{{ q}_i}, 
\frac{{\bar I}+{q}_i}{{ q}_i}\right).$$
Hence, in this situation,  one can define 
$$e^{\infty}_{HK}(R, I) := e^{\infty}_{HK}({\bar R}, {\bar I}) =
\sum_{i=1}^rl_i e_{HK}^{\infty}\left(\frac{{\bar R}}{{ q}_i}, 
\frac{{\bar I}+{q}_i}{{ q}_i}\right).$$ 
\end{rmk}
In Section~4, we study some properties of $f^{\infty}(R,I)$ (when it exists), and prove that 
$f^{\infty}(R, I)$ behaves well under Segre product (Propositions~\ref{p2} and \ref{p3}). 
In case of nonsingular projective curves 
(Theorem~\ref{vb}), the function $f(R_s, I_s)-f^{\infty}(R, I)$ is 
nonnegative, continous and characterizes the behaviour of the HN 
filtration of the syzygy bundle of the curve, reduction mod $\Char~k(s)$. Hence 
$f((S_1\#\cdots \#S_n)_p)-f^{\infty}(S_1\#\cdots \#S_n) = 0$ if and only if 
the Harder-Narasimhan filtrations of the  syzygy bundles of $\mbox{Proj}~S_i$ are
the strong HN filtrations reduction mod $p$, for all $i$.  

We give an example to show that, for an
arbitrary Segre product of plane trinomial curves, the function 
$f((S_1\#\cdots \#S_n)_p)-f^{\infty}(S_1\#\cdots S_n) = 0$ for a Zariski dense set 
of primes.

It is easy to check  that 
if $f^{\infty}(R,I)$ exists (in $L^{\infty}$ norm) then $e_{HK}^{\infty}(R,I)$ exists.
One can ask the converse, {\em i.e.},

\vspace{5pt}

\noindent{\bf Question}: Does the existence of $e_{HK}^{\infty}(R, I)$ imply the 
existence of $f^{\infty}(R, I)$?

\vspace{5pt}
By Proposition~\ref{p2}, an affirmative answer to this question will 
imply the existence of the 
$e_{HK}^{\infty}$ for Segre products of the rings for which  
$e_{HK}^{\infty}$ exist.

\section{A key proposition}
Throughout this section, $R$  is a Noetherian standard graded integral domain
 of dimension $d$ over an algebraically 
closed field $k$ of $\Char~p >0$, $I$ is 
a homogeneous ideal of $R$ such that $\ell(R/I)< \infty$.
Let $h_1, \cdots, h_{\mu}$ be a set of homogeneous generators of $I$ of 
degrees 
$d_1, \ldots, d_{\mu}$ respectively.

 Let $X = {\rm Proj}~R$; then we have an associated  canonical short exact
 sequence of locally free sheaves of $\sO_X$-modules
(moreover the sequence is locally split exact)

\begin{equation}\label{e2}
0\longto V\longto \oplus_i\sO_X(1-d_i)\longto \sO_X(1)\longto 0,\end{equation}
where $\sO_X(1-d_i)\longto \sO_X(1)$ is given by the multiplication by the 
 element  $h_i$.

For a coherent sheaf $\sQ$ of $\sO_X$-modules, the sequence of $\sO_X$-modules
 $$ 0\longto F^{n*}V\tensor \sQ(m)\longto
 \oplus_i\sQ(q-qd_i+m)\longto   \sQ(q+m)\longto 0$$
is exact as the short exact sequence~(\ref{e2}) is locally 
split as $\sO_X$-modules (as usual, $q=p^n$ and $F^n$ is the $n^{th}$ iterate 
of the absolute Frobenius morphism).
Therefore we have  a long exact sequence of cohomologies
\begin{equation}\label{*}0\longto H^0(X, F^{n*}V\tensor \sQ(m))\longto
 \oplus_iH^0(X, \sQ(q-qd_i+m))\by{\phi_{m,q}(\sQ)} 
H^0(X, \sQ(q+m))\end{equation}

$$ \longto H^1(X, F^{n*}V\tensor\sQ(m))\longto \cdots,$$
for $m\geq 0$ and $q = p^n$.

We recall the definition of (Castelnuovo-Mumford) regularity.

\begin{defn}\label{d2}Let $\sQ$ be a coherent sheaf of $\sO_X$-modules and 
let $\sO_X(1)$ be a very ample line bundle on $X$. We say that
 $\sQ$ is ${\tilde m}$-regular (or ${\tilde m}$ is a regularity number of 
$\sQ$) with respect to $\sO_X(1)$, if
for all $m\geq {\tilde m}$
\begin{enumerate}
\item the canonical multiplication map
$H^0(X, \sQ(m))\tensor H^0(X, \sO_X(1))\longto 
H^0(X, \sQ(m+1))$ is surjective and
\item $H^i(X, \sQ(m-i)) = 0$, for $i\geq 1$.\end{enumerate} 
\end{defn}

\begin{notations}\label{n2}
\begin{enumerate}
\item Let 
$$P_{(R,{\bf m})}(m) = {\tilde e_0}{{m+d-1}\choose{d}}
-{\tilde e_1}{{m+d-2}\choose{d-1}}+\cdots +(-1)^{d}{\tilde e_{d}}$$
be the Hilbert-Samuel polynomial of $R$ with respect to the graded maximal
 ideal ${\bf m}$. Therefore 
$$\chi(X, \sO_X(m)) = {\tilde e_0}{{m+d-1}\choose{d-1}}
-{\tilde e_1}{{m+d-2}\choose{d-2}}+\cdots +(-1)^{d-1}{\tilde e_{d-1}}.$$
\item Let ${\bar m}$ be a positive integer such that 
\begin{enumerate}
\item ${\bar m}$ is a regularity number for $(X, \sO_X(1))$, and 
\item $R_m = h^0(X,\sO_X(m))$, for all $m\geq {\bar m}$. In particular
$\ell(R/{\bf m}^m) = P_{(R, {\bf m})}(m)$, for all $m\geq {\bar m}$. 
\end{enumerate}
\item Let $l_1= h^0(X, \sO_X(1))$ and 
\item let $n_0\geq 1$ be an integer such that 
$R_{n_0}\subseteq I$.
\item
We  denote $\dim_k{\rm Coker}~\phi_{m,q}(\sQ)$ by ${\rm coker}~\phi_{m,q}(\sQ)$.
\end{enumerate}
\end{notations}

\begin{rmk}\label{r1}
\begin{enumerate}
 \item  The canonical map 
$\oplus_mR_m \longto \oplus_mH^0(X, \sO_X(m))$ is injective, as $R$ is an integral domain.
\item For  $m+q \geq m_R(q) = {\bar m}+ n_0(\sum_id_i)q$, we have 
  $\coker~\phi_{m,q}(\sO_X) = \ell(R/I^{[q]})_{m+q} = 0$:
 Because 
$m_R(q) =  {\bar m} + n_0\mu q + n_0(\sum_i(d_i-1))q \implies 
 q-qd_i+m\geq {\bar m},$
 for all $i$.
Hence the map $\phi_{m, q}(\sO_X)$ is the map 
$\oplus_iR_{q-qd_i+m}\longrightarrow R_{m+q}$,
 where the map $R_{q-qd_i+m}\rightarrow R_{m+q}$ is given by 
 multiplication by
the element $h_i^q$.
Therefore, $\coker~\phi_{m,q}(\sO_X) = \ell(R/I^{[q]})_{m+q}$. Moreover, 
by Lemma~\ref{l1}, we have $\ell(R/I^{[q]})_{m+q} = 0$, as ${m+q} \geq 
{\bar m}+n_0\mu q$.

\item For  $C_R = (\mu)h^0(X, \sO_X({\bar m}))$, we have  
\begin{equation}\label{ee1}
|\coker~\phi_{m,q}(\sO_X) - \ell(R/I^{[q]})_{m+q} |\leq C_R,\end{equation}
for all $n, m\geq 0$ and $q = p^n$: Because

\vspace{5pt}
\noindent{if} $m+q < {\bar m}$, then
$$|\coker~\phi_{m,q}(\sO_X) - \ell(R/I^{[q]})_{m+q} |\leq h^0(X, \sO_X(m+q))
 \leq h^0(X, \sO_X({\bar m})).$$
If $m+q \geq {\bar m}$, then $h^0(X, \sO_X(m+q))= \ell(R_{m+q})$ and therefore 
$$|\coker~\phi_{m,q}(\sO_X) - \ell(R/I^{[q]})_{m+q}| \leq 
\sum^{\mu}_{i=1}|h^0(X, \sO_X(q-qd_i+m)) -\ell(R_{q-qd_i+m})|.$$
Now, if $q-qd_i+m < {\bar m}$ then 
$\ell(R_{q-qd_i+m})\leq h^0(X, \sO_X(q-qd_i+m)) \leq h^0(X, \sO_X({\bar m})$,
 and if 
  $q-qd_i+m \geq {\bar m}$ then $R_{q-qd_i+m} =  H^0(X, \sO_X(q-qd_i+m))$.

Hence 
 $$|\coker~\phi_{m,q}(\sO_X) - \ell(R/I^{[q]})_{m+q} |\leq
 \mu h^0(X, \sO_X({\bar m})).$$
\end{enumerate}
\end{rmk}

\begin{defn}\label{d1}Let $\sQ$ be a coherent sheaf of 
$\sO_X$-modules of dimension ${\bar d}$ and let ${\tilde m}\geq 1$ be the 
least integer which is  
 a regularity number for $\sQ$ with respect to $\sO_X(1)$. 
Then we define $C_0(\sQ)$ and $D_0(\sQ)$ as follows:
Let $a_1, \ldots, a_{\bar d}\in H^0(X,\sO_X(1))$ be such that we have a short 
exact sequence of $\sO_X$-modules 
$$0\rightarrow \sQ_i(-1)\by{a_i} \sQ_i \rightarrow \sQ_{i-1}\rightarrow 0,
\quad\mbox{for}\quad 0< i\leq {\bar d},$$
where $\sQ_d = \sQ$ and $\sQ_i = \sQ/(a_{{\bar d}}, \ldots,a_{i+1})\sQ$, 
for $0\leq i<{\bar d}$, with $\dim~\sQ_i = i$.
We define 
$$C_0(\sQ) = \mbox{min}\{\sum_{ i=0}^{\bar d}h^0(X,\sQ_i)\mid
 a_1,\ldots, a_{\bar d}~~~\mbox{is a}~~~\sQ-\mbox{sequence as above}\},$$
$$D_0(\sQ) = h^0(X,\sQ({\tilde m})) + 
2({\bar d} +1) \left(\mbox{max}\{q_0, q_1, \ldots, q_{{\bar d}}\}\right),$$
where
$$\chi(X, \sQ(m)) = q_0{{m+{\bar d}}\choose{{\bar d}}}
- q_1{{m+{\bar d}-1}\choose{{\bar d}-1}}+\cdots +(-1)^{\bar d}
 q_{\bar d}$$
is the Hilbert polynomial of $\sQ$.
\end{defn}

A more general version of the  following lemma has been  stated and proved 
in Lemma~2.5 of  [T4]. 
Here we state and prove a relevant part of it, for a self-contained treatment 
(avoiding additional complications).

\begin{lemma}\label{l2} Let $\sQ$ be a coherent sheaf of $\sO_X$-modules 
of dimension ${\bar d}$.  Let $P$ be a locally-free sheaf of $\sO_X$-modules
 which fits into a short exact sequence of locally-free
 sheaves of $\sO_X$-modules
\begin{equation}\label{e15}0\longto P\longto  \oplus_{i} \sO_{X}(-b_i)
\longto P''\longto 0,~~~~\mbox{where}~~~ b_i 
\geq 0.\end{equation}
Then, for ${\tilde \mu} = \Rank(P)+ \Rank(P'')$ and, for all $n, m \geq 0$, 
 we have $$h^0(X, \sQ(m+q))\leq D_{0}(\sQ)(m+q)^{\bar 
d}\quad\mbox{and}\quad h^0(F^{n*}P\tensor \sQ(m)) \leq ({\tilde 
\mu}){C_{0}}(\sQ)(m^{{\bar d}}+1).$$ \end{lemma} 

\begin{proof}Let 
${\tilde m}$ be  a regularity number for $\sQ$, then by Definition~\ref{d1}, 
we have  
$$h^0(X, \sQ(m+q))\leq D_{0}(\sQ)(m+q)^{\bar d},\quad\mbox{for all}\quad 
n, m \geq 0.$$

Let $\sQ_{{\bar d}} = \sQ$. Let 
$a_{{\bar d}}, \ldots, a_1\in H^0(X, \sO_X(1))$ with the exact 
sequence of $\sO_X$-modules 
$$
0\longto {\sQ_i}(-1)\longby{a_i} \sQ_i\longto \sQ_{i-1}\longto 0,$$
where $\sQ_i = \sQ_{{\bar d}}/(a_{{\bar d}}, \ldots, a_{i+1})\sQ_{{\bar d}}$,
for $0 \leq  i \leq {\bar d}$, and realizing the minimal value $C_0(\sQ)$. 
Now, by the exact sequence~(\ref{e15}),  we have the following
short exact 
sequence of $\sO_X$-sheaves
$$0\longto F^{n*}P\tensor \sQ_i \longto 
\oplus_j \sQ_i(-qb_j) \longto F^{n*}P''\tensor\sQ_i\longto 0.$$
This implies 
$H^0(X, F^{n*}P\tensor\sQ_i) \into 
 \oplus_jH^0(X, \sQ_i(-qb_j)).$
Therefore 
\begin{equation}\label{e11}h^0(X, F^{n*}P\tensor\sQ_i) 
\leq \sum_jh^0(X, \sQ_i(-qb_j)) \leq ({\tilde \mu})h^0(X, \sQ_i),\end{equation}
as $-b_j \leq  {0}$.
Since $F^{n*}P$ is a locally-free sheaf of $\sO_X$-modules, we have 
$$0\longto F^{n*}P\tensor{\sQ_i}(m-1)\longby{a_i} F^{n*}P\tensor\sQ_i(m)\longto
F^{n*}P\tensor\sQ_{i-1}(m)\longto 0,$$
which is a short exact sequence of $\sO_X$-sheaves. 
Now by induction on $i$, we prove that, for $m\geq 1$, 
$$h^0(X, F^{n*}P\tensor \sQ_i(m))\leq ({\tilde \mu})\left[h^0(X, \sQ_{i})+\cdots+
h^0(X, \sQ_0)\right](m^{i}).$$
For $i =0$, the inequality holds as 
 $h^0(X, F^{n*}P\tensor \sQ_0(m))\leq ({\tilde \mu})h^0(X, \sQ_0)$ 
(as $\dim~\sQ_0 =0$).

Now, for $m\geq 1$, by the inequality~\ref{e11} and by induction on $i$, we have 
$$h^0(X, F^{n*}P\tensor \sQ_i(m))  \leq h^0(X, F^{n*}P\tensor \sQ_i)
 + h^0(X, F^{n*}P\tensor \sQ_{i-1}(1))+ \cdots + 
h^0(X, F^{n*}P\tensor \sQ_{i-1}(m))$$
$$ \quad\quad\quad\quad\quad\quad\quad\quad\leq ({\tilde \mu})h^0(X, \sQ_i)+{\tilde \mu}
\left[h^0(X, \sQ_{i-1})+\cdots+
h^0(X, \sQ_0)\right](1+2^{i-1}+\cdots+m^{i-1})$$
$$ \leq ({\tilde \mu})
\left[h^0(X, \sQ_{i})+\cdots+
h^0(X, \sQ_0)\right]m^{i}.\quad\quad\quad\quad$$
This implies 
$$h^0(X, F^{n*}P\tensor \sQ(m)) = h^0(X, F^{n*}P\tensor \sQ_{\bar d}(m)) 
\leq {\tilde \mu}C_0(\sQ)m^{\bar d},$$ 
for all $m\geq 1$. Therefore, for all $0\leq i\leq {\bar d}$
 and for all $m\geq 0$, 
we have $h^0(X, F^{n*}P\tensor \sQ(m))\leq 
{\tilde \mu}C_0(\sQ)(m^{\bar d}+1)$.
This proves the lemma.\end{proof}

\begin{lemma}\label{l4}There exists a short exact sequence of coherent 
sheaves of $\sO_X$-modules 
$$0\longto \oplus^{p^{d-1}}\sO_X(-d)\longto F_*\sO_X\longto \sQ\longto 0,$$ 
where $\sQ$ is a coherent sheaf of 
$\sO_X$-modules such that $\dim~\supp(\sQ) < d-1$. \end{lemma} 
\begin{proof}Note that $X = {\rm Proj}~R$, where 
 ${R} = \oplus_{n\geq 0}R_n$,  
is a standard graded domain such that $R_0$ is an algebraically closed field.
 Therefore there exists a Noether 
normalization $$k[X_0, \ldots, X_{d-1}] \longto {R},$$ which is an 
injective,
 finite separable graded map of degree $0$ (as $k$ is an algebraically 
closed field). This induces a finite separable affine map 
$\pi:X\longrightarrow \P^{d-1}_k = S$.

Note that there is also an isomorphism 
$$\eta: \sO_S^{\oplus n_0}\oplus \sO_S(-1)^{\oplus n_1}\oplus
\cdots \oplus \sO_S(-d+1)^{\oplus n_{d-1}}\longto F_*\sO_S$$ 
of 
$\sO_S$-modules, where $\sum n_i = p^{d-1}$. 

Now the isomorphism of $\eta$ implies that the map ${\pi^*(\eta)}: 
\oplus_{i=0}^{d-1}\sO_X(-i)^{\oplus n_i}\longto \pi^*F_*\sO_s$ is an 
isomorphism of $\sO_X$-sheaves. Since $0\leq i \leq d-1$, we also have 
an injective and generically isomorphic map of $\sO_X$-sheaves 
$$\oplus^{p^{d-1}}\sO_X(-d) \longto \oplus_{i=0}^{d-1}\sO_X(-i)^{\oplus 
n_i}.$$ Composing this map with $\pi^*(\eta)$ gives an injective and 
generically isomorphic map of $\sO_X$-sheaves $${\alpha}: 
\oplus^{p^{d-1}}\sO_X(-d) \longto \pi^*F_*\sO_S.$$

Since $\pi$ is separable, there is a canonical map 
$\beta: \pi^*F_*\sO_S \longto 
F_*\sO_X$, of sheaves 
of $\sO_X$-modules, which is generically isomorphic.

Now we have the composite map
 $$\beta\circ \alpha:\oplus^{p^{d-1}}\sO_X(-d)\longrightarrow
 \pi^*F_*\sO_S \rightarrow F_*\sO_X $$ 
which is generically an isomorphism. Hence 
$\dim{\rm Coker}(\beta\circ\alpha) < \dim~X =d-1$ and  the map
 ${\beta\circ\alpha}: \oplus^{p^{d-1}}\sO_X(-d)\longto F_*\sO_X $ is 
injective, as $X$ is an integral scheme.
 This proves the lemma. 
\end{proof}

\begin{lemma}\label{l5}Let $$0\longto \oplus^{p^{d-1}}\sO_X(-d)\longto
 F_*\sO_X\longto \sQ\longto 0$$
as in  Lemma~\ref{l4}. Then
\begin{enumerate}
\item $\sQ$ is ${\tilde m}$-regular, where ${\tilde m} = 
\mbox{max}\{{\bar m}+d, l_1-1\}$, where ${\bar m}$ and $l_1$ 
are as given in Notations~\ref{n2}.
\item For a given $d$, there exists a universal polynomial function 
${\bar P^d_{1}}(X_0, \ldots, X_{d-1}, Y)$
with rational coefficients  
(and hence independent of $p$) such that 
$$2C_0(\sQ) + D_0(\sQ) \leq 
p^{d-1}{\bar P^d_{1}}({\tilde e_0}, {\tilde e_1}, \ldots, 
{\tilde e_{d-1}}, {\bar m}).$$
\end{enumerate}
\end{lemma}
\begin{proof}(1) The above short exact sequence of $\sO_X$-sheaves gives a long 
exact sequence of cohomologies
$$ \oplus^{p^{d-1}}H^i(X, \sO_X(m-d))\longto
H^i(X, \sO_X(mp))\longto 
H^i(X, \sQ(m))
\longto \oplus^{p^{d-1}}H^{i+1}(X, \sO_X(m-d)).$$
But $h^i(X, \sO_X(m-d-i)) = 0$, for all $m \geq {\bar m}+d$ and $i\geq 1$, 
which implies that if  $m \geq {\bar m}+d$ then 
$h^i(X, \sQ(m-i)) = 0$, for  $i\geq 1$, and the canonical  map
$$f_{1,m}:H^0(X, (F_*\sO_X)(m))\longto H^0(X, \sQ(m))$$ 

is surjective. Moreover the canonical map 
$$H^0(X, (F_*\sO_X)(m))\tensor H^0(X, \sO_X(1)) \longto 
H^0(X, (F_*\sO_X)(m+1))$$ is
$$f_{2,m}: H^0(X, \sO_X(mp))\tensor H^0(X, \sO_X(1))^{[p]}\longto 
H^0(X, \sO_X(mp+p)),$$
is surjective for $m\geq {\tilde m}$ because it fits into the following 
 canonical diagram
$$\begin{array}{ccc}
R_{mp}\tensor R_1^{[p]}& \longto & R_{mp+p}\\
\downarrow & {} & \downarrow\\
H^0(X, \sO_X(mp))\tensor H^0(X, \sO_X(1))^{[p]}& \longby{f_{2,m}} & 
H^0(X, \sO_X(mp+p))\end{array}$$
where the top horizontal map is surjective for 
 $m\geq l_1-1$. 
Now the following commutative diagram of canonical maps 
$$\begin{array}{ccc}
H^0(X, (F_*\sO_X)(m))\tensor H^0(X, \sO_X(1)) & \longto  & 
H^0(X,\sQ(m))\tensor H^0(X,\sO_X(1))\\
\downarrow{f_{2,m}} & & \downarrow\\
H^0(X, (F_*\sO_X)(m+1)) & \longby{f_{1,m+1}} & 
H^0(X, \sQ(m+1))\end{array}$$
implies that the second vertical map is surjective, 
for $m\geq {\tilde m}$, as the maps $f_{2, m}$ and $f_{1,m+1}$
 surjective. This proves that $\sQ$ is ${\tilde m}$-regular. Hence the
 assertion (1). 

\vspace{5pt}
\noindent{(2)}\quad If 
\begin{equation}\label{e12}
\chi(X, \sQ(m)) = {q_0}{{m+d-2}\choose{d-2}}
-{q_1}{{m+d-3}\choose{d-3}}+\cdots +(-1)^{d-2}{q_{d-2}},\end{equation}
then by Lemma~\ref{l9} (in the Appendix, below),
$$|q_i| \leq p^{d-1}P^d_i({\tilde e_0}, \ldots, {\tilde e_{i+1}}),$$
where  $P^d_i(X_0, \ldots, X_{i+1})$ is a universal 
polynomial 
function with rational coefficients. 

Now, $\sQ$ is ${\tilde m}$-regular implies that, for $0\leq i < d$,  $\sQ_i:= 
\sQ/(a_{\bar d}, \ldots, a_{i+1})\sQ$ is ${\tilde m}$-regular, 
for any $\sQ$-sequence $a_1, \ldots, a_{\bar d} \in H^0(X, \sO_X(1))$.
Therefore 
$$\begin{array}{lcl}
 h^0(X,\sQ_i)& \leq & h^0(X, \sQ_i({\tilde m}))\leq h^0(X, \sQ_{i+1}({\tilde m}))
\leq \cdots\leq h^0(X, \sQ({\tilde m}))= \chi(X, \sQ({\tilde m}))\\
{} & \leq & \displaystyle{|{q_0}|{{{\tilde m}+d-2}\choose{d-2}}
+|{q_1}|{{{\tilde m}+d-3}\choose{d-3}}+\cdots + |{q_{d-2}}|}.\end{array}$$
This implies 
$h^0(X, \sQ_i) 
\leq p^{d-1}P^d({\tilde e_0}, \ldots, {\tilde e_{d-1}}, {\tilde m})$,
 where $P^d(X_0, \ldots, X_{d-1}, Y)$ is a universal
 polynomial function with rational coefficients.
Therefore 
$$C_0(\sQ) \leq 
(d-1)p^{d-1}P^d({\tilde e_0}, \ldots, {\tilde e_{d-1}},{\tilde m}).$$
The inequality for 
$D_0(\sQ)$ follows similarly. This proves the assertion~(2) and hence the 
lemma.\end{proof}

\begin{lemma}\label{l6} Given $d\geq 2$, there exist universal polynomials 
${\bar P_2^d}$, ${\bar P_3^d}$ $\in \Q[X_0, \ldots, X_{d-1}, Y]$ such that, 
if $X$ is an integral variety of dimension $d-1$ 
with Hilbert-polynomial and ${\bar m}$ as in Notations~\ref{n2}, and if 
there are short exact sequences of $\sO_X$-modules 
$$0\longto \sO_X(-m_0)\longto \sO_X \longto Y_1\longto 
0\quad\mbox{and}\quad 0\longto \sO_X\longto \sO_X(n_2) \longto Y_2\longto 0$$
then 
$$2C_0({Y_1})+D_0({Y_1}) \leq m_0^{d-1}
{\bar P_2^d}({\tilde e_0},\ldots, {\tilde e_{d-1}}, {\bar m}),$$
$$2C_0({Y_2})+D_0({Y_2}) \leq n_2^{d-1}{\bar P^d_3}({\tilde e_0},\ldots,
 {\tilde e_{d-1}}, {\bar m}),$$
where $m_0\geq 0$ and $n_2\geq 0$ are two integers.
 \end{lemma}

\begin{proof}Without loss of generality one can assume $m_0\geq 1$ 
and $n_0\geq 1$.
Since $\sO_X$ is ${\bar m}$-regular, the sheaf 
${Y_1}$ is ${\bar m}+m_0$-regular sheaf of 
$\sO_X$-modules of dimension $d-2$. Therefore, for any
 ${Y_1}$-sequence $a_1, \ldots, a_{d-2}$, the sheaf ${Y_{1i}}:=
{Y_1}/(a_{d-2}, \ldots, a_{i+1}){Y_1}$ is ${\bar m}+m_0$-regular, as 
$\sO_X$-modules.
This implies
$$\begin{array}{lcl}
h^0(X, {Y_{1i}}) &\leq & h^0(X, Y_{1i}({\bar m}+m_0)) \leq 
h^0(X, Y_1({\bar m}+m_0)) \leq h^0(X, \sO_X({\bar m}+m_0))\\
 & \leq & \displaystyle{m_0^{d-1}
 \left[{\tilde e_0}^2{{{\bar m}+1+d-2}\choose{d-1}}
+{\tilde e_1}^2{{{\bar m}+1+d-1}\choose{d-2}}+\cdots +{\tilde e_{d-1}}^2\right]}\\
 & \leq & {m_0}^{d-1}{\tilde P^d}({\tilde e_0}, \ldots, {\tilde e_{d-1}}, 
{\bar m}),\end{array}$$
 where ${\tilde P^d}(X_0, \ldots, X_{d-1}, Y)$ is 
a universal polynomial function with rational coefficients.

Let $e_i(Y)$ denote the  $i^{th}$ coefficient of the Hilbert polynomial of 
$(Y_1, \sO_X(1))$.
Then by Lemma~\ref{l9}, we have 
 $e_i({Y_1}) \leq m_0^{i+1}P^d_i({\tilde e_0}, \ldots, {\tilde e_{i}})$,
 where $P^d_i(X_0, \ldots, X_i)$ is a universal 
 polynomial with rational coefficients.
 
Now the bound  for $2C_0({Y_1})+D_0({Y_1})$ follows. The identical 
proof follows for ${Y_2}$. \end{proof} 

\begin{notations}\label{n1} For a pair $(R, I)$, where $R$ is a standard
graded ring of $\Char~p >0$ and of dimension $d\geq 2$, 
 we define (similar to  the sequence of functions which  had been defined in 
[T4]), 
a sequence of functions 
$\{f_n:[1, \infty)\rightarrow [0, \infty)\}_n$, as follows: 

Fix $n\in \N$ and denote $q = p^n$.
Let $ x\in [0,\infty)$ then there exists a unique nonnegative integer $m$ such that 
$(m+q)/q \leq x < (m+q+1)/q$. 
Then $$f_n(x) = 1/q^{d-1}\ell(R/I^{[q]})_{m+q}.$$ 
\end{notations}

\begin{lemma}\label{l1}Each $f_n:[1, \infty)\longto [0, \infty)$, defined as 
in Notations~\ref{n1},  is a compactly supported  
function such that 
$\cup_{n\geq 1}\mbox{Supp}~f_n \subseteq [1, n_0\mu]$, where 
$R_{n_0} \subseteq I$ and $\mu \geq \mu(I)$.
\end{lemma}
\begin{proof}
Since $R$ is a standard graded ring, for $m\geq n_0\mu q $, we have $R_m \subseteq (R_{n_0})^{\mu q}
\subseteq I^{\mu q} \subseteq I^{[q]}$.
 This implies
$\ell(R/I^{[q]})_{m} = 0$, if $m \geq  n_0\mu q$.
Therefore support~$f_n \subseteq [1, n_0\mu]$, for every $n\geq 0$. 
This proves the lemma. \end{proof}

\vspace{10pt}
\begin{propose}\label{p1} For $f_n$ as given in Notations~\ref{n1}, we have 
\begin{enumerate}
\item $|f_n(x)-f_{n+1}(x)| \leq C/p^{n-d+2},$ for every $x\in [1, \infty)$,
and for all  $n\geq 0$.
 \item In particular, $||f_n-f_{n+1}|| \leq C/p^{n-d+2}$ and 
$||f_{d-1}-f_d|| \leq C/p,$\end{enumerate}
 where 
\begin{equation}\label{e30} C = 2C_R+
\mu\left({\bar m}+n_0(\sum_{i=1}^{\mu} d_i)+1\right)^{d-2}({\bar P^d_1}+
d^{d-1}{\bar P^d_2}+
{\bar P^d_3})\end{equation}
and the integers ${\bar m}$ and $n_0$
 are given as in Notations~\ref{n2}, and $d_1, \ldots, d_{\mu}$
 are degrees of a chosen generators of $I$.
Moreover $C_R  = \mu h^0(X, \sO_X({\bar m}))$, for $X = {\rm Proj}~R$, and 
${\bar P^d_1}$, ${\bar P^d_2}$ and ${\bar P^d_3}$ are given as in
Lemma~\ref{l5} and Lemma~\ref{l6}.
 \end{propose}
\begin{proof}Fix $x\in [1, \infty)$.
Therefore, for given $q = p^n$,
 there exists a unique integer $m\geq 0$, such that
$ (m+ q)/q \leq x < (m+q+1)/q $
and
$$\frac{(m+q)p+n_2}{qp}
 \leq x < \frac{(m+q)p+n_2+1}{qp},~~~~~~\mbox{for some}~~~ 0\leq n_2 < p.$$
Hence
$$f_{n}(x) = \frac{1}{q^{d-1}}\ell\left(\frac{R}{I^{[q]}}
\right)_{m+q}~~~\mbox{and}~~~~~
f_{n+1}(x) = \frac{1}{(qp)^{d-1}}\ell\left(\frac{R}{I^{[qp]}}\right)_{mp+qp+n_2}.$$

Now, by  Equation~(\ref{ee1}) in Remark~\ref{r1}, we have

\begin{equation}\label{e31}\left|f_n(x)-\frac{\coker~\phi_{m,q}(\sO_X)}{q^{d-1}}\right| < 
\frac{C_R}{q^{d-1}}~~~\mbox{and}~~~  
\left|f_{n+1}(x)-\frac{\coker~\phi_{mp+n_2,qp}(\sO_X)}{(qp)^{d-1}}\right| < 
\frac{C_R}{(qp)^{d-1}}.\end{equation}
Consider the short exact sequence of $\sO_X$-modules
$$0\longto \sO_X(-d)\longto \sO_X \longto Y_1\longto 0.$$
Then, for any locally free sheaf $P$ of $\sO_X$-modules and for $m\geq 0$,
 we have the following short exact 
sequence of $\sO_X$-modules
$$0\longto F^{n*}P\tensor\sO_X(-d+m)\longto F^{n*}P\tensor\sO_X(m) 
\longto F^{n*}P\tensor Y_1(m)\longto 0.$$
Since 
$$\coker~\phi_{m,q}(\sO_X) = h^0(X, \sO_X(m+q))-
\sum_ih^0(X, \sO_X(m+q-qd_i))+h^0(X, (F^{n*}V)(m))$$
we have (by taking $P =V$ and $= \sum \sO_X(1-d_i))$ 
respectively), 
$$|\coker~\phi_{m,q}(\sO_X)-\coker~\phi_{m-d,q}(\sO_X)|
\quad\quad\quad\quad\quad\quad\quad\quad $$
$$ \leq h^0(X, Y_1(m+q)) +  h^0(X,\sum_i\sO_X(q-qd_i)\tensor Y_1(m))+
h^0(X, F^{n*}V\tensor Y_1(m)). $$
If $d-2=0$ then 
$$|\coker~\phi_{m,q}(\sO_X)-\coker~\phi_{m-d,q}(\sO_X)| 
\leq (1+\mu-1+\mu)h^0(X, Y_1) = 2\mu C_0(Y_1).$$
If $d-2 >0$ then 
 by Lemma~\ref{l2}  
$$\begin{array}{lcl}
|\coker~\phi_{m,q}(\sO_X)-\coker~\phi_{m-d,q}(\sO_X)| & \leq & 
 D_0(Y_1)(m+q)^{d-2} + 2(\mu) C_0(Y_1)(m^{d-2}+1)\\
&  \leq &
(\mu)\left[2C_0(Y_1)+D_0(Y_1)\right](m+q)^{d-2}.\end{array}$$

Therefore
\begin{equation}\label{e8}
|p^{d-1}\coker~\phi_{m,q}(\sO_X)-p^{d-1}\coker~\phi_{m-d,q}(\sO_X)|
\leq (\mu)\left[2C_0(Y_1)+D_0(Y_1)\right](m+q)^{d-2}p^{d-1}.\end{equation}

Since, for a locally free sheaf $P$, we have 
$$h^0(X, F^{n*}P\tensor (F_*\sO_X)(m)) = 
h^0(X, F^{(n+1)*}P\tensor \sO_X(mp)),$$
 the short exact sequence in the statement of Lemma~\ref{l5} gives a 
canonical long exact sequence 
$$0\longto \oplus H^0(X,(F^{n*}P)(m-d))\longto H^0(X,(F^{(n+1)*}P)(mp))
\longto H^0(X, F^{n*}P\tensor\sQ(m))\longto \cdots, $$
which implies  
\begin{equation}\label{e9}
 |p^{d-1}\coker~\phi_{(m-d), q}(\sO_X) - 
\coker~\phi_{mp, qp}(\sO_X)|\leq 
(\mu)\left[2C_0(\sQ)+D_0(\sQ)\right]
(m+q)^{d-2}.
\end{equation}
The short exact sequence of $\sO_X$-modules
$$0\longto \sO_X\longto \sO_X(n_2) \longto Y_2\longto 0$$ gives
$$0\longto H^0(X,(F^{(n+1)*}P)(mp))
\longto H^0(X,(F^{(n+1)*}P)(mp+n_2))
\longto H^0(X,(F^{(n+1)*}P)\tensor Y_2(mp)),$$
which gives 
$$\begin{array}{l}
{|\coker~\phi_{mp,qp}(\sO_X)-\coker~\phi_{mp+n_2,qp}(\sO_X)|}\\
{}\\
\leq h^0(X, F^{(n+1)*}V\tensor Y_2(mp))
+ h^0(X,\sum_i\sO_X(qp-qpd_i)\tensor Y_2(mp))+h^0(X, Y_2(mp+qp))\\
{}\\
\leq 2(\mu) C_0(Y_2)((mp)^{d-2}+1)+(\mu) D_0(Y_2)(mp+qp)^{d-2}.\end{array}$$
Therefore
\begin{equation}\label{e10}|\coker~\phi_{mp,qp}
(\sO_X)-\coker~\phi_{mp+n_2,qp}(\sO_X)|\leq 
(\mu)\left[2C_0(Y_2)+D_0(Y_2)\right](mp+qp)^{d-2}.\end{equation}
Combining Equations~(\ref{e8}), (\ref{e9}) and (\ref{e10}),
  we get 
$$\begin{array}{l}
(A) := |p^{d-1}\coker~\phi_{m, q}(\sO_X) - \coker~\phi_{mp+n_2, qp}(\sO_X)|\\
{}\\
\leq (\mu)(m+q)^{d-2}\left[(2C_0(Y_1)+D_0(Y_1))p^{d-1}+(2C_0(\sQ)+D_0(\sQ))
+(2C_0(Y_2)+D_0(Y_2))p^{d-2}\right]\\
{}\\
\leq (\mu)(m+q)^{d-2}\left[p^{d-1}d^{d-1}{\bar P^d_2}
+p^{d-1}{\bar P^d_1} + n_2^{d-1}{\bar P^d_3}p^{d-2}\right],
\end{array}$$
where the last inequality follows from Lemma~\ref{l5},  Lemma~\ref{l6} with 
${\bar P^d_i} = {\bar P^d_i}({\tilde e_0}, \ldots, {\tilde e_{d-1}}, 
{\tilde m})$, for $i=1,2$ and $3$.

Now, as $n_2< p$, we have 
$$(A) \leq (\mu)(m+q)^{d-2}\left[p^{d-1}d^{d-1}{\bar P^d_2}
+p^{d-1}{\bar P^d_1}+ p^{d-1}p^{d-2}{\bar P^d_3}\right].$$

Now multiplying the above inequality by $1/(qp)^{d-1}$ we get
$$|\frac{\coker~\phi_{m, q}(\sO_X)}{q^{d-1}} - 
\frac{\coker~\phi_{mp+n_2, qp}(\sO_X)}{(qp)^{d-1}}| 
 \leq (\mu)\frac{(m+q)^{d-2}}{q^{d-1}}\left[d^{d-1}{\bar P^d_2}
+{\bar P^d_1}+ p^{d-2}{\bar P^d_3}\right].$$
Moreover, by Remark~\ref{r1},
$$m+q \geq {\bar m} + n_0(\sum_id_i)q + q \implies  
\coker~\phi_{m, q}(\sO_X) = 
\coker~\phi_{mp+n_2, qp}(\sO_X) = 0.$$
Also
$$m+q \leq   {\bar m} + n_0(\sum_id_i)q + q, \implies (m+q)^{d-2}\leq 
L_0q^{d-2},~~~\mbox{where}~~~~L_0 = ({\bar m} + n_0(\sum_id_i) + 1)^{d-2}.$$ 
Therefore, for every $m\geq 0$ and $n\geq 1$, where $q = p^n$, we 
have
$$\begin{array}{lcl}
\displaystyle{|\frac{\coker~\phi_{m, q}(\sO_X)}{q^{d-1}} - 
\frac{\coker~\phi_{mp+n_2, qp}(\sO_X)}{(qp)^{d-1}}|} 
& \leq & \displaystyle{(\mu)\frac{L_0q^{d-2}}{q^{d-1}}\left[d^{d-1}{\bar P^d_2}
+{\bar P^d_1}+ p^{d-2}{\bar P^d_3}\right]}\\
 & \leq & \displaystyle{(\mu)L_0\left[d^{d-1}{\bar P^d_2}
+{\bar P^d_1}+ {\bar P^d_3}\right]\frac{p^{d-2}q^{d-2}}{q^{d-1}}}.\end{array}$$
Now by  Equation~\ref{e31}, we have 
$$|f_n(x)-f_{n+1}(x)|\leq \frac{C_R}{q^{d-1}}+\frac{C_R}{(qp)^{d-1}}+
(\mu)L_0\left[d^{d-1}{\bar P^d_2}
+{\bar P^d_1}+ {\bar P^d_3}\right]\frac{p^{d-2}}{q} \leq C\frac{p^{d-2}}{q},$$ 
where 
 $C = 2C_R + (\mu)L_0\left(d^{d-1}{\bar P^d_2}+{\bar P^d_1}+
{\bar P^d_3}\right)$, which proves the proposition.\end{proof}

\begin{cor}\label{r4} There exists 
a constant  $C_1 = P^d_4({\tilde e_0}, {\tilde e_1}, \ldots, 
{\tilde e_{d}}, {\bar m})  + (n_0\mu-1)C$, where $C$ is 
as in Proposition~\ref{p1} and $P^d_4(X_0, \ldots, X_d, Y)$ is a universal 
polynomial function with rational coefficients such that, 
for $n\geq 1$ 
 $$\left|\frac{1}{(p^n)^d}\ell\left(\frac{R}{I^{[p^n]}}\right) - 
\frac{1}{(p^{n+1})^d}\ell\left(\frac{R}{I^{[p^{n+1}]}}\right)\right|\leq 
\frac{C_1}{p^{n-d+2}}.$$
\end{cor}
\begin{proof}Note
$$\sum_{m=0}^{\infty}\frac{1}{p^{nd}}\ell\left(\frac{R}{I^{[p^n]}}\right)_{m+q} = \int_{1}^{\infty}
f_n(x)dx = \int_{1}^{n_0\mu} f_n(x)dx,$$
where the last equality follows from Lemma~\ref{l1}.
Hence 
$$\left|\frac{1}{(p^n)^d}
\ell\left(\frac{R}{I^{[p^n]}}\right) - 
\frac{1}{(p^{n+1})^d}\ell\left(\frac{R}{I^{[p^{n+1}]}}\right)\right|$$
$$ \leq 
\left|\frac{1}{p^{nd}}\ell\left(\frac{R}{{\bf m}^{p^n}}\right) - 
 \frac{1}{p^{(n+1)d}}\ell\left(\frac{R}{{\bf m}^{p^{n+1}}}\right)\right|+
\left|\int_1^{n_0\mu}f_n(x)dx -
\int_1^{n_0\mu}f_{n+1}(x)dx\right|.$$
If $p^n\leq {\bar m}$ then 
$$\left|\frac{1}{p^{nd}}\ell\left(\frac{R}{{\bf m}^{p^n}}\right) - 
 \frac{1}{p^{(n+1)d}}\ell\left(\frac{R}{{\bf m}^{p^{n+1}}}\right)\right|\leq 
\left|\frac{P_{(R,{\bf m})}({\bar m})}{p^{nd}}-\frac{P_{(R, 
{\bf m})}({\bar m}^2)}{p^{(n+1)d}}\right| \leq 
\frac{P_{(R, {\bf m})}({\bar m}^2)}{p^n},$$
if  $p^n > {\bar m}$ then there exists a universal polynomial function 
 $P^d_6(X_0, \ldots, X_d)$
with rational coefficients
such that 
$$\mbox{L.H.S.} \leq \left|\frac{P_{(R,{\bf m})}(p^n)}{p^{nd}}-\frac{P_{(R, 
{\bf m})}(p^{n+1})}{p^{(n+1)d}}\right| \leq 
\frac{P^d_6({\tilde e_0}, {\tilde e_1}, \ldots, 
{\tilde e_{d}})}{p^n}.$$
Therefore, combining this with Proposition~\ref{p1}, part~(1), we
get a  universal polynomial function 
 $P^d_4(X_0, \ldots, X_d,Y)$
with rational coefficients
such that 
$$\left|\frac{1}{(p^n)^d}\ell\left(\frac{R}{I^{[p^n]}}\right) - 
\frac{1}{(p^{n+1})^d}\ell\left(\frac{R}{I^{[p^{n+1}]}}\right)\right| 
\leq \frac{P^d_4({\tilde e_0}, {\tilde e_1}, \ldots, 
{\tilde e_d}) }{p^n} + \frac{(n_0\mu-1)C}{p^{n-d+2}}.$$
Since $d\geq 2$, the corollary follows.
 \end{proof}

\section{Hilbert-Kunz density function and reduction mod $p$}
\begin{rmk}\label{r3}
Let $R$ be a standard graded  integral domain of dimension $d\geq 2$,
 with $R_0 = k$, where $k$ is an algebraically closed field.
Let $N = \ell(R_1)-1$, then we have  a surjective
 graded map $k[X_0, \ldots, X_N] \longto R$ of degree $0$, 
 given by $X_i$ mapping to generators of $R_1$.
This gives a closed immersion 
$X= {\rm Proj}~R \longto \P^N_k$ such that 
$\sO_X(1) = \sO_{\P^N_k}(1)\mid_X$.
Therefore if 
 
$$P_R(m)) = {\tilde e_0}{{m+d-1}\choose{d}}
-{\tilde e_1}{{m+d-2}\choose{d-1}}+\cdots +(-1)^{d}{\tilde e_{d}}$$
 is the Hilbert-Samuel polynomial of $R$ then the Hilbert polynomial for 
the pair $(X, \sO_X(1))$ is 
$$\chi(X, \sO_X(m)) = {\tilde e_0}{{m+d-1}\choose{d-1}}
-{\tilde e_1}{{m+d-2}\choose{d-2}}+\cdots +(-1)^{d-1}{\tilde e_{d-1}}.$$
Since $R$ is a domain, the canonical graded map 
$R = \oplus_m R_m \longto \oplus_mH^0(X,\sO_X(m))$ is  injective.

Let $\sI_X$ be the ideal sheaf of $X$ in $\P^N_k$; then we have 
the canonical short exact sequence of $\sO_{\P^N_k}$-modules 
$$0\longto \sI_X\longto \sO_{\P^N_k}\longto \sO_X\longto 0$$
and the image of 
the induced map $f_m:H^0({\P^N_k}, \sO_{\P^N_k}(m))\longto H^0(X, \sO_X(m))$
is $R_m$.
Now, by Exp~XIII, (6.2) (in [SGA6]), there exists a universal polynomial 
$P^d_5(t_0, \ldots, t_{d-1})$ with rational coefficients such that 
the sheaf $\sI_X$ is 
${\bar m} = P^d_5({\tilde e_0}, \ldots, 
{\tilde e_{d-1}})$-regular. Therefore the map $f_m$ is surjective
 for $m\geq {\bar m}$.

In particular, we have 
\begin{enumerate}
\item $R_m = H^0(X,\sO_X(m))$, for all $m\geq {\bar m}$ 
and 
\item the sheaf $\sO_X$ is ${\bar m}$-regular with respect
 to  $\sO_X(1).$
\end{enumerate}
\end{rmk}

Next we recall a notion of {\em spread}.
\begin{defn}\label{d3} Consider the pair $(R, I)$, where
$R$ is
 a finitely generated $\Z_{\geq 0}$-graded $d$-dimensional domain such that 
$R_0$ is an
algebraically closed
 field $k$ of characteristic $0$, and $I\subset R$ is
a
homogeneous ideal of finite colength. For  such a pair, there exists a
finitely
generated $\Z$-algebra $A\subseteq k$, a finitely generated $\N$-graded
algebra $R_A$
over $A$ and a homogeneous ideal $I_A\subset R_A$ such that
$R_A\tensor_Ak = R$ and 
$I = {\rm{Im}}(I_A\tensor_A {k})$. We call $(A, R_A, I_A)$ a {\it spread} of
the pair $(R, I)$.

Moreover, if, for the pair $(R, I)$,  we have a
{spread} $(A,R_A,I_A)$
as above and
 $A\subset A' \subset
k$, for some finitely generated $\Z$-algebra $A'$  then $(A', R_{A'},
I_{A'})$
satisfy the same properties as $(A, R_A, I_A)$. Hence we may always assume
that the spread $(A, R_A, I_A)$  as above is
chosen
such that $A$ contains a given finitely generated $\Z$-algebra $A_0 \subseteq
k$.\end{defn}

\begin{notations}\label{n6} Given a spread $(A, R_A, I_A)$ as above, for a 
closed point $s\in\Spec(A)$, we define
 $R_s = R_A\tensor_A {\bar k(s)}$ and 
the ideal $I_s = {\rm{Im}}(I_A\tensor_A {\bar k(s)})
 \subset R_s $. 
Similarly for $X_A:={\rm Proj}~R_A$, we define
$X_s:= X\tensor {\bar k(s)} = {\rm Proj}~R_s$ and, for a coherent sheaf  
$V_A$ on $X_A$, we define $V_s = V_A\tensor {\bar k(s)}$.

\end{notations}

\begin{rmk}\label{spread} Note that for a spread $(A, R_A, I_A)$ of $(R, I)$ 
as above, the induced map 

${\tilde \pi}:X_A:={\rm Proj}~R_A \longto
 \Spec(A)$ is a proper map, hence by generic flatness 
 there is an open set 
(infact non empty as $A$ is an integral domain) $U\subset \Spec(A)$
 such that 
${\tilde \pi}\mid_{{\tilde \pi}^{-1}(U)}:{{\tilde \pi}^{-1}(U)} \longto U$
 is a proper flat
 map. Therefore (see [EGA IV]~12.2.1) the set 
$$\{s\in \Spec(A)\mid X\tensor_{\Spec(A)} \Spec(k(s))
\quad\mbox{is geometrically integral}\}$$
is a nonempty open set of $\Spec(A)$.
Hence replacing $A$ by a finitely generated $\Z$-algebra $A'$ such 
that $A \subset A'\subset k$ (if necessary) we can 
assume that ${\tilde \pi}$ is a flat map such that for every $s\in \Spec(A)$,
the fiber over $s$ is geomerically integral.

Therefore for any closed point $s\in \Spec(A)$ ({\it
i.e.}, a maximal ideal of $A$) the ring $R_s$ is a 
standard graded  $d$-dimensional
ring such that  the ideal $I_s 
 \subset R_s $ is a homogeneous ideal of finite colength. Moreover
$X_s$ is an integral scheme  over
${\bar {k(s)}}$.\end{rmk} 
\vspace{10pt} 

\noindent{PROOF of Theorem~\ref{t2}}\quad For given $(R, I)$, and a given  
 spread $(A, R_A, I_A)$, we can choose a spread $(A', R_{A'}, I_{A'})$,
 where $A\subset A'$, such that 
the induced projective morphism of Noetherian schemes  
${\tilde \pi}:X_{A'}\longto A'$ is flat and, for every $s\in \Spec(A')$, $X_s$
 is an integral scheme over ${\bar k(s)}$ of dimension
 $= d-1$.
Let $R_{A'}=\oplus_{i\geq 0}(R_{A'})_i$ and 
let $(R_{A'})_1$ be generated by $N$ elements as an $A'$-module. Then
 the canonical graded surjective map
 $$A'[X_0, \ldots, X_N] \longto R_{A'},$$  
gives a closed immersion 
$X_{A'}= {\rm Proj}~R_{A'} \longto \P^N_{A'}$ such that 
$\sO_{X_{A'}}(1) = \sO_{\P^N_{A'}}(1)\mid_{X_{A'}}$.
Let $X_s = X_{A'}\tensor {\bar k(s)}$.  Then $X_s = {\rm Proj}~R_s$ and 
$\sO_{X_s}(1)$ is the canonical line bundle induced by 
$ \sO_{X_{A'}}(1)$.
 Let $s_0 = \Spec{Q(A)} = \Spec{Q(A')}$ be the generic point of $\Spec(A')$. 
We now have the following, 
\begin{enumerate}
\item The Hilbert polynomial for 
the pair $(X_s, \sO_{X_s}(1))$ is 
$$\chi(X_s, \sO_{X_s}(m)) = {\tilde e_0}{{m+d-1}\choose{d-1}}
-{\tilde e_1}{{m+d-2}\choose{d-2}}+\cdots +(-1)^{d-1}{\tilde e_{d-1}},$$
where the coefficients ${\tilde e_i}$ are as above for $(X,\sO_X(1))$ (from 
$\Char~0$).
 
In particular, $\dim~X_s =  d-1$ and
\item by Remark~\ref{r3}, there exists 
${\bar m} = P^d_5({\tilde e_0}, \ldots,
 {\tilde e_{d-1}})$ 
 such that $(R_s)_m = H^0(X_s, \sO_{X_s}(m))$ for all
 $m\geq {\bar m}$ and $(X_s, \sO_{X_s}(1))$ is ${\bar m}$-regular.
\item Moreover, by the semicontinuity theorem (Chapter~III, Theorem~12.8 in 
[H]), by  shrinking $\Spec(A')$ further, we have 
 $h^i(X_s, \sO_{X_s}({\bar m}))$ and $h^0(X_s,\sO_{X_s})$ is 
independent of $s$, 
for all  $i\geq 0$.
\item Again by shrinking $\Spec(A')$ (if necessary), can choose $n_0\in\N$ such
that $R_{A'}^{n_0} \subseteq I_{A'}$. This implies $R_s^{n_0}\subseteq I_s$, 
for all $s\in \Spec(A')$.
\end{enumerate}

Let  $s\in \Spec(A')$ and let $p = \Char~k(s)$.
We sketch the proof of the existence of the map $f(R_s, I_s):[1, \infty)\to \R$ 
and its relation 
to $e_{HK}(R, I)$ (note that we have proved this in a more general setting in [T4]).
By Proposition~\ref{p1}, for any given $s$, 
the sequence $\{f^s_n\}_n$ of functions is 
uniformly convergent. Let $f(R_s, I_s)(x) = \lim_{n\to \infty}f_n(R_s, I_s)(x)$. 
This implies
 that $\lim_{n\to \infty}\int_{1}^{\infty}f_n(R_s, I_s)(x) = 
\int_{1}^{\infty}f(R_s, I_s)(x)$
 as, by Lemma~\ref{l1}, there is a compact set containing 
$\cup_n {\rm supp}~f_n(R_s, I_s)$.
On the other hand 
$$\begin{array}{lcl}
e_{HK}(R_s, I_s) & = & \displaystyle{\lim_{n\to \infty} \frac{1}{p^{nd}}
\ell\left(\frac{R_s}{I_s^{[p^n]}}\right)
= \lim_{n\to \infty} \frac{1}{p^{nd}}\ell\left(\frac{R_s}{{\bf m}_s^{p^n}}
\right) + \lim_{n\to \infty}\frac{1}{p^{nd}}\sum_{m\geq 0}\ell
\left(\frac{R}{I^{[p^n]}}\right)_{m+p^n}}\\
 & =  &\displaystyle{e(R_s, {\bf m}_s) + \lim_{n\to \infty}
\int_{1}^{\infty}f_n(R_s, I_s)(x)
=e(R_s, {\bf m}_s) + \int_{1}^{\infty}f(R_s,I_s)(x)},\end{array}$$
where ${\bf m}_s$ is the graded maximal ideal of $R_s$. 

Now, by 
Proposition~\ref{p1}, there exists a constant  
$$C= 2C_{R_s}+\mu\left({\bar m}+n_0(\sum_{i=1}^{\mu}d_i) 
+1\right)^{d-2}({\bar P^d_1}+d^{d-1}{\bar P^d_2}+
{\bar P^d_3}),$$
which is independent of the choice of  $s$ in $\Spec(A')$ (as $C_{R_s} = \mu h^0(X_s, \sO_{X_s}(1))$)
such that 
$$||f_n(R_s, I_s)-f_{n+1}(R_s, I_s)||\leq C/p^{n-d+2},\quad\mbox{for all}\quad n.$$

In particular,  for given $m\geq d-1$, 
$$||f_{m}(R_s,I_s)-f(R_s,I_s)||\leq \left[C/p+C/p^2+C/p^3+\cdots 
\right]\frac{1}{p^{m-(d-1)}}
\leq \frac{2C}{p^{m-d+2}}.$$
As $s\to s_0$ we have $\Char~k(s) \to \infty$, which implies
$\lim_{s\to s_0}||f_{m}(R_s, I_s)-f(R_s,I_s)|| = 0$.
This proves the first assertion of the theorem.

Since each $f_{m}(R_s, I_s)$ and $f(R_s, I_s)$ has support in the compact interval 
$[1, n_0\mu]$, the above inequality implies that, for any fixed $m\geq d-1$, 
$$\lim_{s\to s_0}\left|\int_{1}^{\infty}f_{m}(R_s, I_s)(x)dx-
\int_{1}^{\infty}f(R_s, I_s)(x)dx\right| \leq \hspace{10cm} $$
$$   
\lim_{s\to s_0}\int_{1}^{\infty}|f_{m}(R_s, I_s)(x)-f(R_s, I_s)(x)|dx \leq 
\lim_{s\to s_0} \left(\frac{2C}{p^{m-d+2}}\right)(n_0\mu-1) = 0.$$
Moreover it is easy to see that 
$$\lim_{s\to s_0}\left[\frac{1}{p^{md}}
\ell\left(\frac{R_s}{{\bf m}_s^{p^{m}}}\right)
-e(R_s, {\bf m}_s)\right] = 0.$$
Therefore 
$$\lim_{s\to s_0}\left[\frac{1}{p^{md}}
\ell\left(\frac{R_s}{I_s^{[p^{m}]}}\right)-
e_{HK}(R_s, I_s)\right] = $$
$$ \lim_{s\to s_0}\left[\left\{\frac{1}{p^{md}}
\ell\left(\frac{R_s}{{\bf m}_s^{p^{m}}}\right) + 
\int_{1}^{\infty}f_{m}(R_s, I_s)(x)dx\right\} - \left\{e(R_s, {\bf m}_s) +
\int_{1}^{\infty}f(R_s, I_s)(x)dx\right\} \right]=  0.$$ $\Box$

Now the proof of  Corollary~\ref{c1} is obvious.

\section{Some properties and examples}
Throughout this section  $R$ is a standard graded  integral domain of 
dimension $d\geq 2$,
 with $R_0 = k$, where $k$ is an algebraically closed field of characteristic $0$, and 
$I\subset R$ is a homogeneous ideal of finite colength.
Our choice of  spread satisfies conditions as given in Remark~\ref{spread}.
\begin{defn}\label{d5} We denote 
$f^{\infty}(R, I)=\lim_{s\to s_0}f(R_s, I_s)$ if it exists, where
for $(R, I)$ the pair  $(R_s, I_s)$ is given as in Definition~\ref{d3} 
and Notations~\ref{n6}.
\end{defn}

\begin{defn}\label{d2} For a choice of 
spread $(A, R_A, I_A)$ of $(R, I)$ , as in Remark~\ref{spread}, 
 and a closed point $s\in \Spec(A)$, we define 
$$HSd(R_s)(x) = F_{R_s}(x) = \lim_{n\to \infty}F_n(R_s)(x),~~~~\mbox{where}~~~
F_n(R_s)(x) =
\frac{1}{q^{d-1}}\ell({R_s}_{\lfloor xq \rfloor})~~\mbox{and}~~q=p^n.$$
One can check that
 $$F_{R_s}:\R\rightarrow \R~~\mbox{is given by}~~~
F_{R_s}(x) = 0,~~~\mbox{for}~~~ x < 0,~~~~\mbox{and}~~~
F_{R_s}(x) = e_0(R, {\bf m})x^{d-1}/(d-1)!,~~~~\mbox{ for}~~~ x\geq 0,$$
   where
$e_0(R,{\bf m})$ is the Hilbert-Samuel multiplicity of $R$ with respect to
${\bf m}$.
Hence we denote $F_{R_s}(x) = F_R(x)$. 
Moreover, for any $n\geq 1$ we have $\lim_{s\to s_0}F_n(R_s)(x) = F_{R_s}(x) = F_R(x)$.
 \end{defn}

\begin{propose}\label{p2}Let $R$ and $S$ be standard graded 
domains, where  $R_0 = S_0 = k$, where $k$ is an algebraically closed field of 
characteristic $0$   
with $I\subset R$ and $J\subset S$ homogeneous ideals of finite
 coelength respectively.
If $f^{\infty}(R, I)(x)$ and $f^{\infty}(S, J)(x)$ exist then 
$f^{\infty}(R\# S, I\# J)(x)$ exists and 
$$ f^{\infty}(R\# S, I\# J)(x) = F_S(x)f^{\infty}(R, I)(x) +
 F_R(x)f^{\infty}(S, J)(x) - f^{\infty}(R, I)(x)f^{\infty}(S, J)(x).$$
In particular $f^{\infty}(-, -)$ satisfies a multiplicative formula on Segre products.
\end{propose}
\begin{proof} Let us denote  $f^{\infty} = f^{\infty}(R,I)$ and 
$g^{\infty} = g^{\infty}(S,J)$. For $q=p^n$, where $p = \mbox{char~k(s)}$, 
 we denote $f_n^s = f_n(R_s, I_s)$ and $g_n^s 
= f_n(S_s, J_s)$, where $s\in \Spec(A)$ denotes a closed point and $(A, R_A,I_A)$ and 
$(A, S_A, J_A)$ are spreads.

For any $n\geq 1$, we have
$$f_n(R_s\#S_s, I_s\#J_s)(x) = F_n(R_s)(x)g_n^s(x) + F_n(S_s)(x)f_n^s(x) -
f_n^s(x)g_n^s(x).$$

For a spread $(A, R_A, I_A)$,  let $n_0$ and $\mu $ be positive integers such that 
 ${(R_A)_1}^{n_0}\subseteq I_A$, ${(S_A)_1}^{n_0}\subseteq J_A$ and 
$\mu(I_A)$, $\mu(J_A)\leq  \mu $.  Then, by Lemma~\ref{l1},  
$$\bigcup_{n\geq 0, s\in\Spec(A)}\mbox{Support}~(f_n^s) \bigcup
\bigcup_{n\geq 0, s\in\Spec(A)}\mbox{Support}~(g_n^s) \subseteq 
[0, n_0\mu].$$
Moreover, there is a constant $C_1$ such that, for any $n\geq 1$ and every closed point $s\in \Spec(A)$, we have  
$$f_n^s(x)\leq F_n(R_s)(x)\leq C_1~~~\mbox{and}~~~g^s_n(x)\leq F_n(S_s)(x)
\leq C_1,$$
for all $x\in [0, n_0\mu]$. 

Since $f^{\infty}$ and $g^{\infty}$ exists,
 by Theorem~\ref{t2}~(1), for given $n\geq d_1+d_1-2$, 
$\lim_{s\to s_0} f_n^s = f^{\infty}$
 and $\lim_{s\to s_0} g_n^s = g^{\infty}$.
This implies, for given 
$n\geq d_1+d_2-2$, 
we have the limit function computation
$$\lim_{s\to s_0}F_n(R_s)(x)g_n^s(x) +  F_n(S_s)(x)f_n^s(x)
- f_n^s(x)g_n^s(x) = F_R(x)g^{\infty}(x) + F_S(x)f^{\infty}(x) - 
f^{\infty}(x)g^{\infty}(x).$$ 
Hence, for any $n\geq d_1+d_2-2$,  
$$\lim_{s\to s_0}f_n(R_s\#S_s, I_s\#J_s)(x) = 
F_R(x)g^{\infty}(x) + F_S(x)f^{\infty}(x) -f^{\infty}(x)g^{\infty}(x).$$
Now, by Theorem~\ref{t2}~(1), the proposition follows.\end{proof}
 
\begin{propose}\label{p3}Let the pairs $(R,I)$ and $(S, J)$ be as 
in  Proposition~\ref{p2}. Let  $(A, R_A,I_A)$, 
$(A, S_A, J_A)$ be spreads and  $s\in \Spec(A)$  a closed point. Suppose 
 $f(R_s, I_s)\geq f^{\infty}(R, I)$ and $f(S_s, J_s)\geq 
f^{\infty}(S, J)$. Then 
\begin{enumerate}
\item  $f(R_s\#S_s, I_s\#J_s)\geq f^{\infty}(R\#S, I\#J)$. Moreover, 

\item  if $(A, R_A,I_A)$, 
$(A, S_A, J_A)$ be spreads and $s\in \Spec(A)$ is a closed point such that 
 $I_s\cap (R_s)_1\neq 0 $ and $J_s\cap (S_s)_1\neq 0$ 
then 
 $$f(R_s, I_s) = f^{\infty}(R, I)~~\mbox{and}~~ f(S_s, J_s)= f^{\infty}(S, J)~~
\iff f(R_s\#S_s, I_s\#J_s) = f^{\infty}(R\#S, I\#J).$$ 
\end{enumerate}
\end{propose}
\begin{proof} (1)\quad  Let us denote  $f^{\infty} = f^{\infty}(R,I)$ and 
$g^{\infty} = g^{\infty}(S,J)$ and  
  denote $f^s = f(R_s, I_s)$ and $g^s 
= f(S_s, J_s)$.

We know, by the multiplicative property of the HK density functions (see Proposition~2.18 of [T4]), that 
$$\begin{array}{lcl}
f(R_s\#S_s,I_s\#J_s)(x)
& = &  F_R(x)g^s(x) + F_S(x)f^s(x) - f^s(x)g^s(x)\\
&  = &  \left(F_R(x)-f^s(x)\right) g^s(x) +F_S(x) f^s(x)\\
& \geq &  \left(F_R(x)-f^s(x)\right) g^{\infty}(x) +F_S(x) 
f^s(x)\\
& = &  F_R(x)g^{\infty}(x) + f^s(x)\left[F_S(x)-
g^{\infty}(x)\right]\\
& \geq &  F_R(x)g^{\infty}(x) + f^{\infty}(x)
\left[F_S(x)-g^{\infty}(x)\right]\\
& = & f^{\infty}(R\#S, I\#J)(x),\end{array}$$
where $3^{rd}$ and $5^{th}$ inequalities hold as $F_R(x) \geq f^s(x)$ and 
$F_S(x)\geq g^s(x)$, for every  $s\in Spec~A$, and the last equality follows from 
Proposition~\ref{p2}.

\noindent{(2)}\quad Suppose $I$ and $J$ are the ideals of $R$ and $S$ respectively, 
and $s\in \Spec(A)$ is a closed point
such that $I_s\cap (R_s)_1\neq 0$ and $J_s\cap (S_s)_1 \neq 0$ then 
we 

\vspace{2pt}

\noindent{CLAIM}.\quad $F_R(x) > f^s(x)$ and $G_S(x) > g^s(x)$, for all 
$x\geq 1$. 

\vspace{2pt}

\noindent{\underline{Proof of the claim}}:\quad Enough to prove that $F_R(x+1) > f^s(x+1)$, 
for $x>0$. Choose an  integer $n_0$ such that $x\geq 1/p^{n_0}$.
where $p = \mbox{char}~k(s)$. Let $q=p^n$ for 
some $n$. 
For a given nonzero $y\in I_s\cap (R_s)_1$, we have an injective map 
of $R_s$-linear map 
$$\oplus_{m\geq 0} (R_s)_m \longrightarrow \oplus_{m\geq 0}(I_s^{[q]})_{m+q},$$
 of degree $q$, given by the multiplication by element $y^q$.
Therefore $\ell(I_s^{[q]})_{m+q}\geq \ell(R_s)_m$, for all $m\geq 0$.
Since $\lfloor xq\rfloor = m$ if and only if 
$\lfloor (x+1)q\rfloor = m +q$, we have
$\ell(I_s^{[q]})_{\lfloor(x+1)q\rfloor}\geq \ell(R_s)_{\lfloor{xq}\rfloor}$. Hence 

$$\ell(R_s/I_s^{[q]})_{\lfloor(x+1)q\rfloor}\leq
\ell(R_s)_{\lfloor(x+1)q\rfloor} - \ell(R_s)_{\lfloor{xq}\rfloor}.$$

Therefore
$$f_n(R_s, I_s)(x+1) \leq F_n(R_s)(x+1) -F_n(R_s)(x).$$
But 
$$ \lim_{n\to \infty} F_n(R_s)(x)\geq
\frac{1}{(d-1)!}\frac{e(R)}{(p^{n_0})^{d-1}} > 0.$$
This implies 
$f^s(x+1) = f(R_s, I_s)(x+1) < F_R(x+1)$.
This proves the  claim.

Now, retracing the  above argument, we note  that $f(R_s\#S_s, I_s\#J_s) = 
f^{\infty}(R\#S, I\#J)$ if and only if 
$$\left[F_R(x)-f^s(x)\right] g^{s}(x) =  \left[F_R(x)-f^s(x)\right] 
g^{\infty}(x)$$
and
$$\left[F_S(x)-g^{\infty}(x)\right] f^{s}(x) = 
\left[F_S(x)-g^{\infty}(x)\right] f^{\infty}(x).$$ Hence, by  the above 
claim, we have
 $f^s(x) = f^{\infty}(x)$ and $g^s(x) = g^{\infty}(x)$ for all 
$x > 1$.
For $0\leq x \leq 1$, we always have $F_R(x) = f^s(x) = f^{\infty}(x)$
and $F_S(x) = g^s(x) = g^{\infty}(x)$. 
This proves the proposition.
\end{proof}

\begin{ex}\label{ex5} Let $R$ be a two dimensional  
 standard graded normal domain, where $R_0 =k$ is an algebraically closed field of $\Char~0$.
 Let $I\subset R$ be a homogeneous ideal of finite colength
and generated by homogeneous elements, say $h_1,\ldots, h_\mu$ of
positive
degrees $d_1, \ldots, d_\mu$ respectively. Let $X = \rm{Proj}~R$ be 
 the corresponding nonsingular projective curve.

Let $(A, R_A, I_A)$ and $(A, X_A, V_A)$ denote  spreads for $(R,I)$ and 
$(X, V)$ respectively. 
We have a short exact sequence of $\sO_{X_A}$-sheaves 

 Then we have an associated  canonical exact
 sequence of locally free sheaves of $\sO_X$-modules
(moreover the sequence is locally split exact).

\begin{equation}\label{}
0\longto V_A\longto \oplus_i\sO_{X_A}(1-d_i)\longto \sO_{X_A}(1)\longto 0,\end{equation}
Restricting to the fiber $X_s$ we have 
  the following exact
 sequence of locally free sheaves of $\sO_{X_s}$-modules.

\begin{equation}\label{e2}
0\longto V_s\longto \oplus_i\sO_{X_s}(1-d_i)\longto \sO_{X_s}(1)\longto 0,\end{equation}

Moreover, we can choose a spread $(A, X_A, V_A)$ such that  there is a filtration
 $$ 0 = E_{0A} \subset E_{1A} \subset \cdots \subset E_{lA} \subset E_{l+1A} = 
V_A,$$
 of locally free sheaves of $\sO_{X_A}$-modules 
such that 
 $$ 0 = E_{0s} \subset E_{1s} \subset \cdots \subset E_{ls} \subset E_{l+1s} =
V_s$$
is the Harder-Narasimhan filtration of the vector bundles over $X_s$ for $s\in \Spec~A$.
 \end{ex}

\vspace{5pt} 

\begin{thm}\label{vb}Let $(R, I)$, $(A, R_A,I_A)$ and 
$(A, X_A, V_A)$ be given as above. Then, for every closed point 
$s\in \Spec(A)$, we have 
\begin{enumerate}
\item $f(R_s, I_s) \geq f^{\infty}(R, I)$ and 
\item $f(R_s, I_s) = f^{\infty}(R, I)$ if and only 
if  the filtration 
$$ 0 = E_{0s} \subset E_{1s} \subset \cdots \subset E_{ls} \subset E_{l+1s} \subset
V_s$$
 is the strongly
 semistable HN filtration of $V_s$  on $X_s$,
{\it i.e.}, 
$$ 0 = F^{n*}E_{0s} \subset F^{n*}E_{1s} \subset \cdots \subset 
F^{n*}E_{ls} \subset F^{n*}E_{l+1s} =
F^{n*}V_s $$ is the HN filtration of $F^{n*}V_s$.
\end{enumerate}
\end{thm}

\begin{proof}
We fix such an $s\in \Spec~A$ and denote the HN filtration of $V_s$ by 
 $$ 0 = E_{0} \subset E_{1} \subset \cdots \subset E_{l} \subset E_{l+1} \subset
V_s$$

By Theorem~2.7 of [L], there is 
$n_1\geq 1$ be such that $F^{n_1*}V_s$ has the strong HN filtration 
(note $n_1$ may depend on $s$).

Then, by Lemma~1.8 of [T2], for $\Char~k(s) > 4({\rm genus}(X_s))\rank(V_s)^3$, the HN filtration of 
$F^{n_1*}V_s$ is
$$0=E_{00}\subset E_{01} \subset \cdots \subset E_{0t_0}\subset 
F^{n_1*}E_1 = E_{10}\subset \cdots \subset $$
$$E_{i-1(t_{i-1}+1)} = F^{n_1*}E_i = E_{i0} \subset E_{i1}\subset \cdots 
\subset E_{it_i}
\subset E_{i(t_i+1)} =  F^{n_1*}E_{i+1} = E_{i+1,0}\subset \cdots \subset F^{n_1*}V_s.$$
Let, for $i\geq 0$ and $j\geq 1$,
$$ a_{ij} = \frac{1}{p^{n_1}}\mu\left({E_{ij}}/{E_{i,{j-1}}}\right),
~~\mbox{and}~~~r_{ij} = \rank(E_{ij}/E_{i, (j-1)}).$$
Let $$\mu_0 = 1~~\mbox{and for}~~i\geq 1~~\mbox{let}~~~\mu_i = \mu(E_i/E_{i-1})~~~
\mbox{and}~~~r_i = \rank(E_i/E_{i-1}).$$
Note that, for any $i\geq 1$, the only possible inequalities are  
$$a_{(i-1)1}\geq \mu_i \geq a_{i-1,t_{i-1}+1}>\ldots  > a_{i0} \geq \mu_{i+1} 
\geq a_{it_i+1}.$$ 

By Lemma~1.14 of [T2], for a given $i$,  
\begin{equation}\label{evb}a_{ij} = \mu_{i+1}+O(1/p),\end{equation} 
where, by $O(1/p)$ we mean  $O(1/p)= C/p$, where $|C|$ is  bounded by a constant 
depending only on the degree of $X$ and rank of $V$.
Now it is easy to check the following:

\vspace{5pt}

\noindent{\bf Claim}\quad If $1-a_{ij_0}/d \leq x < 1-a_{i(j_0+1)}/d$, 
for some $i\geq 0$ and $j\geq 1$,  then 
\begin{enumerate}
\item $$\begin{array}{lcl} 
-\left[a_{kj}r_{kj}+d(x-1)r_{kj}
\right] & = & O(\frac{1}{p}), ~~\mbox{for any}~~ 1\leq k\leq t_{i}+1,~~~ \mbox{and}\\
-\left[a_{kj}r_{kj}+d(x-1)r_{kj}\right] & > & 0~~~\mbox{if}~~ k\geq  j_0+1~~~\mbox{and}~~~ \\
 {} & \leq & 0~~~\mbox{ if}~~ k\leq j_0.\end{array}$$
\item $$ -\sum_{k\geq j_0+1}\left[a_{jk}r_{kj}+d(x-1)r_{kj}\right]
\geq -\left[\mu_{i+1}r_{i+1}+d(x-1)r_{i+1}\right].$$
\end{enumerate}

We also recall that, for $x$ as above 
(by Example~3.3 of [T4]), we have 
$$ f(R_s, I_s)(x) =
-\sum_{j\geq j_0+1}\left[a_{ij}r_{ij}+d(x-1)r_{ij}\right] -
\sum_{k\geq i+2, j\geq 1}\left[a_{kj}r_{kj}+d(x-1)r_{kj}
\right]$$ 

Let $x\geq 1$ then $1-\mu_i/d \leq x < 1-\mu_{i+1}/d$, for some $i\geq 0$.
Now there are three possiblilities.
\begin{enumerate}
\item $1-\frac{\mu_i}{d} \leq x < 1-\frac{a_{i-1,(t_{i-1}+1}}{d}$ then $1-\frac{a_{i-1,j_0}}{d} \leq x < 
1-\frac{a_{i-1,(j_0+1)}}{d}$, for some $j_0\geq 0$. Then 
$$f(R_s,I_s) =   
- \sum_{j\geq j_0+1}\left[a_{i-1,j}r_{i-1,j}+d(x-1)r_{i-1,j}\right] -
\sum_{k\geq i+1}\left[\mu_{k}r_{k}+d(x-1)r_{k} 
\right].$$
\item $1-\frac{a_{i-1,(t_{i-1}+1}}{d} \leq x < 1-\frac{a_{i1}}{d}$ then 
$$f(R_s,I_s) =   - \sum_{k\geq i+1}\left[\mu_{k}r_{k}+d(x-1)r_{k} 
\right].$$
\item $1-\frac{a_{i1}}{d} \leq  x < 1-\frac{\mu_{i+1}}{d}$. Then  $1-\frac{a_{ij_0}}{d} \leq 
x < 1-\frac{a_{i(j_0+1)}}{d}$, for some $j_0\geq 0$. Then  
$$ f(R_s, I_s)(x) = - \sum_{j\geq j_0+1}\left[a_{ij}r_{ij}+d(x-1)r_{ij}\right] -
\sum_{k\geq i+2}\left[\mu_{k}r_{k}+d(x-1)r_{k} \right].$$
\end{enumerate}
Hence  $f^{\infty}(R,I) = \lim_{s\to s_0} f(R_s, I_s)$ exists and 
 $$\begin{array}{lcl}
1 \leq x < 1-\mu_1/d & \implies & f^{\infty}(R, I)(x) = 
-\left[\sum_{i\geq 1}\mu_ir_i+d(x-1)r_i\right] \\
1-\mu_i/d  \leq x < 1-\mu_{i+1}/d  & \implies & f^{\infty}(R, I)(x) =
- \left[\sum_{i\geq {i+1}}\mu_ir_i+d(x-1)r_i\right].\end{array}$$
Moreover 
$f(R_s,I_s) \geq f^{\infty}(R, I)$ for $1\leq x < 1-a_{l+1,0}/d$ and 
$f(R_s,I_s) = f^{\infty}(R, I)$ otherwise. This proves part~(1) of the theorem.
\vspace{5pt}

\noindent{(2)}\quad If $V_s$ has stongly semistable HN filtration
then it is obvious that $f(R_s, I_s) = f^{\infty}(R, I)$. Let, as before, 
$n_1$ be such that $F^{n_1*}V$ has a strongly semistable 
HN filtration in the sense of [L], Theorem~2.7.

If the HN filtration of $V_s$ is not strongly semistable then
$$0= F^{n_1*}E_0\subset F^{n_1*}E_1\subset \cdots \subset F^{n_1*}E_l 
\subset F^{n_1*}V$$
is not the HN filtration of $F^{n_1*}V$. Therefore there is an $i\geq 0$ such that 
$$F^{n_1*}E_i = E_{i0} \subset E_{i1}\subset \cdots \subset F^{n_1*}E_{i+1},$$ 
where $E_{i2}\subseteq F^{n_1*}E_{i+1}$. Since $a_{i1} > \mu_{i+1}$, one can choose
$1-a_{i1}/d < x_0 \leq 1-a_{i2}/d$. Now  
$$ f(R_s, I_s)(x) = - \sum_{j\geq 2}\left[a_{ij}r_{ij}+d(x-1)r_{ij}\right] -
\sum_{k\geq i+2}\left[\mu_{k}r_{k}+d(x-1)r_{k} \right].$$
$$ = \left[a_{i1}r_{i1}+d(x-1)r_{i1}\right] -
\sum_{k\geq i+1}\left[\mu_{k}r_{k}+d(x-1)r_{k} \right] > f^{\infty}(R,I).$$
This proves the Theorem.\end{proof}

\begin{cor}\label{e5}Let $C_1 = \mbox{Proj}~S_1,\ldots, C_n = \mbox{Proj}~S_n$ be 
nonsingular projective curves, over a common field of characteristic $0$. Suppose 
each  syzygy bundle 
$V_{C_i}$, given by 
$$0\longto V_{C_i}\longto H^0(C_i, \sO_{C_i}(1))\tensor \sO_{C_i}
\longrightarrow \sO_{C_i}(1)\longto 0, $$ is 
semistable. ({\it e.g.}, if $\deg \sO_{C_i}(1) > 2\mbox{genus}~(C_i)$ then $V_{C_i}$ is 
semistable, see [KR] and Lemma~2.1 of [T6]).
  
Then there is $n_0$ such that for all $p\geq n_0$ we have 
\begin{enumerate}
\item $f((S_1\#\cdots \#S_n)_{p})(x) \geq
f^{\infty}(S_1\#\cdots\#S_n)(x)$ and
\item  $f((S_1\#\cdots \#S_n)_{p})(x) =
f^{\infty}(S_1\#\cdots\#S_n)(x)$, for all $x\in \R$
if and only if  (mod $p$) reduction of the bundle 
 $V_1\boxtensor \cdots \boxtensor V_n$ is strongly semistable 
on $(C_1\times \cdots\times C_n)_p$.
\end{enumerate}
In particular 
\begin{enumerate}
\item $e_{HK}^{\infty}(S_1\#\cdots\#S_n)$ exists and 
$e_{HK}((S_1\#\cdots\#S_n)_p) \geq e_{HK}^{\infty}(S_1\#\cdots\#S_n)$ and
\item $e_{HK}((S_1\#\cdots\#S_n)_p) =  e_{HK}^{\infty}(S_1\#\cdots\#S_n)$ 
if and only if (mod $p$) reduction of the bundle  
$V_1\boxtensor \cdots \boxtensor V_n$ is strongly semistable on
$(C_1\times \cdots\times C_n)_p$,
\end{enumerate}
where the HK density functions and HK multiplicities are considered with respect to the ideal 
${\bf m}_1\#\cdots\#{\bf m}_n$ for the graded maximal ideals ${\bf m}_i \subset S_i$.
\end{cor}
\begin{proof}Proof follows by  Proposition~\ref{p3} and Theorem~\ref{vb}.\end{proof}

\begin{rmk}With the notations and assumptions as in the corollary above, one can easily compute
$f^{\infty}(S_1\#\cdots\#S_n)$, in terms of ranks of $V_i$ and degrees of $C_i$. In particular, if  
$d_1 = \deg~C_1$ and $d_2 = \deg~C_2$ with $r= \rank~V_1 \geq  s= \rank~V_2$ then 
it follows that
$$e_{HK}^{\infty}(S_1\#S_2) = \frac{d_1d_2}{3}+
d_1d_2\left[\frac{1}{2s}+\frac{1}{6s^2}+\frac{1}{6r^2}+\frac{s}{6r^2}\right].$$
\end{rmk}

\begin{notations}\label{n5}
Let $R= k[x,y,z]/(h)$ be a plane trinomial curve of degree $d$, {\it i.e.}, 
$h = M_1+M_2+M_3$
where $M_i$ are monomials of degree $d$. As given in Lemma~2.2 of [Mo2], 
one can divide such an $h$ in two types:
\begin{enumerate}
\item $h$ is irregular, {\it i.e.}, one of the points 
$(1,0,0)$, $(0,1,0)$, $(0,0,1)$
of $\P^2$ has multiplicity $\geq d/2$ on the plane curve $h$. Here we define 
$\lambda_R = 1$. 
\item $h$ is regular and hence is one of the following type (upto a 
change of variables):
\begin{enumerate}
\item $h = x^{a_1}y^{a_2} + y^{b_1}z^{b_2} +z^{c_1}x^{c_2}$, where 
$a_1, b_1, c_1 > d/2$. Here we define 
$\alpha = a_1+b_1-d$, $\beta = a_1+c_1-d$, $\nu = b_1+c_1-d$ and 
$\lambda = a_1b_1+a_2c_2-b_1c_2$.
\item $h = x^d+x^{a_1}y^{a_2}z^{a_3} + y^bz^c$, where $a_2, c > d/2$. Here we define 
$\alpha = a_2$, $\beta = c$, $\nu = a_2+c-d$ and 
$\lambda = a_2c-a_3b$.
\end{enumerate} 
We denote $\lambda_h = \lambda/a$, where $a = \mbox{g.c.d.}~(\alpha, \beta, \nu, \lambda)$.

\end{enumerate} 
\end{notations}

\begin{cor}\label{c5}Let $S_1, \ldots, S_n$ be a set of irreducible plane trinomial 
curve given by trinomials $h_1, \ldots, h_n$ of degrees $d_1, \ldots, d_n\geq 4$ 
respectively, over a field of characteristic $0$. Then there are spreads
$\{(A_i, S_{iA}), {\bf m}_{iA}\}_i$ such that for every closed point $s\in \Spec(A)$, 
\begin{enumerate}
\item $f^s(S_1\#\cdots \#S_n)(x) = f^{\infty}(S_1\#\cdots \#S_n)(x), 
~~~\mbox{for all}~~~ x\in \R$

if $\Char~k(s)\equiv\pm 1\pmod{\mbox{l.c.m.}(\lambda_{h_1},\ldots, \lambda_{h_n})},$
where $\lambda_{h_i}$ is given as in
Notations~\ref{n5}. Moreover
\item if one of the curve, say, $S_1$ is given by a symmetric trinomial
$h_1 = x^{a_1}y^{a_2}+y^{a_1}z^{a_2}+z^{a_1}x^{a_2}$ such that $d\neq 5$,  then
$$f^s(S_1\#\cdots \#S_n)(x_0) > f^{\infty}(S_1\#\cdots \#S_n)(x_0)\quad\mbox{if}\quad
\Char~k(s)\equiv\pm l\pmod{\lambda_{h_1}}, $$
for some $x_0\in \R$ and for some $l\in (\Z/\lambda_{h_1}\Z)^*$. \end{enumerate}
\end{cor}
\begin{proof}We can choose spreads $(A, S_{iA})$ such that $\Char~k(s) > 
\mbox{max}\{d_1, \ldots, d_n\}^2$, for every closed point $s\in \Spec(A)$. 
Now for any irreducible plane curve given $S = k[x,y,z]/(h)$. Let $S\longto 
{\tilde S}$ be the normalization of $S$. Then it is a finite graded map of 
degree $0$ and $Q(S)= Q({\tilde S})$ such that ${\tilde S}$ is a finitely 
generated $\N$-graded $2$-dimensional domain over $k$. Now, for pairs $(S,{\bf m})$
and $({\tilde S}, {\bf m}{\tilde S})$, we can choose a spread 
$(A, S_A, {\bf m}_A)$ and $(A, {\tilde S}_A, {\bf m}{\tilde S}_A)$
such that for every closed point $s\in\Spec(A)$, the natural map $S_s = 
S_A\tensor k(s)\longto {\tilde S}_A\tensor k(s)$ is a finite graded map of degree 
$0$. This implies, for every $x\geq 0$
$$\lim_{q\to \infty}\frac{1}{q}\ell
\left(\frac{S_s}{{\bf m}^{[q]}}\right)_{\lfloor xq\rfloor}
= \lim_{q\to \infty}\frac{1}{q}\ell
\left(\frac{{\tilde S}_s}{{\bf m}{\tilde S_s}^{[q]}}\right)_{\lfloor xq\rfloor},$$
as kernel and cokernel of the map ${S_s} \longto {\tilde S}_s$ is $0$-dimensional. Therefore 
$f(S_s, {\bf m}_s) = f({\tilde S}_s, {\bf m}{\tilde S}_s)$ and 
$f^{\infty}(S, {\bf m}) = f^{\infty}({\tilde S}, {\bf m}{\tilde S})$. This 
also implies $e_{HK}(S_s, {\bf m}_s)= e_{HK}({\tilde S}_s, {\bf m}{\tilde S}_s)$ 
(this inequality about $e_{HK}$ can aslo be found in Lemma~1.3 of [Mo1], 
Theorem~2.7 in [WY] and [BCP]). Let $\pi:{\tilde X}_s = \mbox{Proj}~{\tilde S}_s
 \longto X_s= \mbox{Proj}~S_s$ be the induced map. We also choose a 
spread $(A, X_A, V_A)$, where $V_A$ is given by 
$$0\longto V_A\longto \sO_{X_A}\oplus
\sO_{X_A}\oplus\sO_{X_A}\longto \sO_{X_A}(1)\longto 0$$
and gives the syzygy bundle $V_s$ with its HN filtration as given in Example~\ref{ex5}.

This gives a short exact seuence of sheaves of $\sO_{{\tilde X}_s}$-modules 
$$0\longto \pi^*V_s\longto \sO_{{\tilde X}_s}\oplus
\sO_{{\tilde X}_s}\oplus\sO_{{\tilde X}_s}\longto \sO_{{\tilde X}_s}(1)\longto 0.$$
Moreover ${\tilde X}_s$ is a nonsingular curve. If $S$ is regular trinomial given by $h$ 
 then, 
by Theorem~5.6 of [T5], the bundle ${\pi}^*(V_s)$ is a strongly semistable, provided 
$\Char~k(s)\equiv\pm 1\pmod{2\lambda_{h_s}}$. Therefore, by Theorem~4.6, we have  
$f({\tilde S}_s, {\bf m}{\tilde S}_s) = f^{\infty}({\tilde S}, {\bf m}{\tilde S})$.
This implies $f(S_s, {\bf m}_s) = f^{\infty}(S, {\bf m})$, for   
$\Char~k(s)\equiv\pm 1\pmod{2\lambda_{S_s}}$.

If $S$ is an irregular trinomial then, by Theorem~1.1 of [T5], $\pi^*V$ 
has a HN filtration $0\subset \sL\subset \pi^*V$. Therefore 
 $0\subset \sL_s\subset \pi^*V_s$ is the HN filtration and hence 
the strong HN filtration (as $\rank~V =2$), for $\pi^*V_s$, for every closed point 
$s\in \Spec~A$. In particular, by Theorem~\ref{vb}, 
$f(S_s, {\bf m}_s) = f^{\infty}(S, {\bf m})$, for all such $s$. Now 
assertion~(1) follows by Proposition~\ref{p3}~(2).

If $S_1 = k[x,y,z]/(h_1)$, where $h_1$ is as in statement~(2) of the corollary then 
$V_s$ (here $X_s$ itself is nonsingular) is semistable, but not strongly semistable, 
if $\Char~k(s)\equiv\pm 2\pmod{\lambda_{{h_1}_s}}$. In particular, by Corollary~\ref{e5},
$f({S_1}_s, {\bf m}{S_1}_s) > f^{\infty}({S_1}_s, {\bf m}{S_1}_s)$, for such $s$. 
Therefore, the statement~(2) follows from Proposition~4.4~(2).\end{proof}

\section{Appendix}

\begin{lemma}\label{l9} For an integer $d\geq 2$, there exist universal
 polynomials
$P_i^d$, $P_i'^d$ in $\Q[X_0, \ldots, X_i]$, where $0\leq i\leq d-2$, 
 such that 
 if for a pair $(X, \sO_X(1))$, where $X$ is an integral variety of 
char $p>0$ and dimension $d-1$
 and $\sQ$ is a coherent sheaf of $\sO_X$-modules with 
the following respective Hilbert polynomials 

$$\chi(X, \sO_X(m)) = {\tilde e_0}{{m+d-1}\choose{d-1}}
-{\tilde e_1}{{m+d-2}\choose{d-2}}+\cdots +(-1)^{d-1}{\tilde e_{d-1}}$$
 and
$$\chi(X, \sQ(m)) = {q_0}{{m+d-2}\choose{d-2}}
-{q_1}{{m+d-3}\choose{d-3}}+\cdots +(-1)^{d-2}{q_{d-2}},$$
then
\begin{enumerate}
\item for $0\leq i\leq d-2$,
$$|q_i| \leq  p^{d-1} P_i^d({\tilde e_0}, \ldots, {\tilde e_{i+1}}),$$
if
there is a short exact sequence of $\sO_X$-modules
 $$0\longto \oplus^{p^{d-1}}\sO_X(-d)\longto F_*\sO_X \longto \sQ\longto 0.$$
\item for $0\leq i\leq d-2$,
 $$|q_i| \leq   m_0^{i+1}P_i^{'d}({\tilde e_0}, \ldots, {\tilde e_{i}}),$$
if  $\sQ$ fits in the  short exact sequence
 $$0\longto \sO_X(-m_0)\longto \sO_X\longto \sQ\longto 0$$
or in the  short exact sequence
$$0\longto \sO_X\longto \sO_X(m_0)\longto \sQ\longto 0$$
of $\sO_X$-modules.

\end{enumerate}

\end{lemma}
\begin{proof}Assertion~(1):\quad Note that for $m\in \Z$, we have
\begin{equation}\label{e13}
\chi(X, \sQ(m)) =
\chi(X, O_X(mp)) -p^{d-1}\chi(X, \sO_X(m)).\end{equation}
We can express, for $1\leq n\leq d-1$,
$$(Y+n)\cdots (Y+2)(Y+1) = \sum_{j=0}^nC^n_jY^j,$$
 where $C^n_{n} = 1$ and, for $j < n$,
$$C_j^n \in
\left\{\sum x_1^{i_1}\cdots x_n^{i_n}\mid {\sum i_l =  n-j},
 \quad 0\leq j< n \leq d-1,\quad \{x_1,\ldots, x_n\} =
 \{1,\ldots, n\}\right\}.$$
Now expanding Equation~(\ref{e13}), we get
$$\frac{\tilde e_0}{(d-1)!}\left[C_{d-2}^{d-1}m^{d-2}(p^{d-2}-p^{d-1})+
C_{d-3}^{d-1}m^{d-3}(p^{
d-3}-p^{d-1})+ \cdots +C_0^{d-1}(1-p^{d-1})\right]$$
$$+ \cdots +\frac{(-1)^i{\tilde e_i}}{(d-1-i)!}\left[C^{d-1-i}_{d-1-i}
m^{d-1-i}(p^{d-1-i}-p^{d-1})+
C_{d-2-i}^{d-1-i}m^{d-2-i}(p^{d-2-i}-p^{d-1})\right.$$
$$\left.+ \cdots +C_0^{d-1-i}(1-p^{d-1})\right]+ \cdots + 
(-1)^{d-1}{\tilde e_{d-1}}\left[(1-p^{d-1})\right]$$

$$= \frac{q_0}{(d-2)!}\left[C^{d-2}_{d-2}m^{d-2}+C^{d-2}_{d-3}m^{d-3}+\cdots +
 C_0^{d-2}\right]-
\frac{q_1}{(d-3)!}\left[C^{d-3}_{d-3}m^{d-3}+C^{d-3}_{d-4}m^{d-4}+\cdots +
 C_0^{d-3}\right]$$
$$+\cdots + \frac{(-1)^{i-1}q_{i-1}}{(d-1-i)!}
\left[C^{d-1-i}_{d-1-i}m^{d-1-i}+C^{d-1-i}_{d-2-i}m^{d-2-i}+\cdots +
\cdots + C_0^{d-1-i}\right] + \cdots + (-1)^{d-2}q_{d-2}.$$
We prove the result for $q_i$, by induction on $i$.
For $i=0$, comparing the coefficients of $m^{d-2}$ on both the sides
we get
$$(p^{d-2}-p^{d-1})\left[\frac{\tilde e_0}{(d-1)!}C^{d-1}_{d-2}-
\frac{\tilde e_1}{(d-2)!}\right]
= \frac{q_0}{(d-2)!}, $$
which implies
$$ |q_0|\leq p^{d-1}\left(|{\tilde e_0}|C^{d-1}_{d-2}+
|{\tilde e_1}|\right)\leq
p^{d-1}\left({\tilde e_0}^2C^{d-1}_{d-2}+
{\tilde e_1}^2\right) .$$
Comparing coefficients of $m^{d-i}$ we get
$$(p^{d-i}-p^{d-1})\left[\frac{\tilde e_0}{(d-1)!}C^{d-1}_{d-i}-
\frac{\tilde e_1}{(d-2)!}C^{d-2}_{d-i}+ \cdots +(-1)^{i-1}
\frac{\tilde e_{i-1}}{(d-i)!}C^{d-1-i}_{d-i}\right]$$
$$= \frac{q_0}{(d-2)!}C^{d-2}_{d-i}-
\frac{q_1}{(d-3)!}C^{d-3}_{d-i}+ \cdots +(-1)^{i}
\frac{q_{i-2}}{(d-i)!}C^{d-i}_{d-i}.$$
This implies that
$$|q_{i-2}|\leq p^{d-1}\left[|{\tilde e_0}|C_{d-i}^{d-1}+|{\tilde e_1}|C_{d-i}^{d-2}+\cdots +
|{\tilde e_{i-1}}|C_{d-i}^{d-1-i}\right] +
\left[|q_0|C_{d-i}^{d-2}+|q_1|C_{d-i}^{d-3}+\cdots +
|q_{i-3}|C_{d-i}^{d+1-i}\right].$$
But
$$p^{d-1}\left[|{\tilde e_0}|C_{d-i}^{d-1}+|{\tilde e_1}|C_{d-i}^{d-2}+\cdots +
|{\tilde e_{i-1}}|C_{d-i}^{d-1-i}\right]\leq
p^{d-1}\left[{\tilde e_0}^2C_{d-i}^{d-1}+{\tilde e_1}^2C_{d-i}^{d-2}+\cdots +
{\tilde e_{i-1}}^2C_{d-i}^{d-1-i}\right].$$

Now the proof follows by induction.

\vspace{5pt}
\noindent{Assertion~(2)}:
For $m_0=0$ the statement is true vacuously.
Therefore we can assume that $m_0\geq 1$.
Now
$$\chi(X, \sQ(m)) = {q_0}{{m+d-2}\choose{d-2}}
-{q_1}{{m+d-3}\choose{d-3}}+\cdots +(-1)^{d-2}{q_{d-2}},$$
$$= \frac{q_0}{(d-2)!}\left[D^{d-2}_{d-2}m^{d-2}+D^{d-2}_{d-3}m^{d-3}+\cdots +
 D_0^{d-2}\right]-
\frac{q_1}{(d-3)!}\left[D^{d-3}_{d-3}m^{d-3}+D^{d-3}_{d-4}m^{d-4}+\cdots +
 D_0^{d-3}\right]$$
$$+\cdots + \frac{(-1)^{i-1}q_{i-1}}{(d-1-i)!}
\left[D^{d-1-i}_{d-1-i}m^{d-1-i}+D^{d-1-i}_{d-2-i}m^{d-2-i}+\cdots +
\cdots + D_0^{d-1-i}\right] + \cdots + (-1)^{d-2}q_{d-2},$$
where
$$D_j^k \in
\left\{\sum x^{i_1}\cdots x_k^{i_k}\mid {\sum i_l =  k-j}\quad
 0\leq j\leq k\leq d-2, \quad \{x_1,\ldots, x_k\} =
 \{1,\ldots, k\}\}\right\}.$$

On the other hand
$$\chi(X, \sO_X(m))-\chi(X, \sO_X(m-m_0)) =
\frac{{\tilde e_0}}{(d-1)!}
\left[C^{d-1}_{d-1}(m_0)(m^{d-2}+\cdots m^{d-3}m_0+\cdots + m_0^{d-2})\right.$$
$$+\left.C^{d-1}_{d-2}(m_0)(m^{d-3}+m^{d-4}m_0+\cdots + m_0^{d-3})
+\cdots + C^{d-1}_1(m_0)\right]
$$
$$-\frac{{\tilde e_1}}{(d-2)!}
\left[C^{d-2}_{d-2}(m_0)(m^{d-3}+\cdots m^{d-4}m_0+\cdots + m_0^{d-3})\right.$$
$$\left.
+C^{d-2}_{d-3}(m_0)(m^{d-4}+m^{d-5}m_0+\cdots + m_0^{d-4})
+\cdots + C^{d-2}_1(m_0)\right]
+\cdots .$$
Again we prove the result for $q_i$, by induction on $i$.
Comparing the coefficients for $m^{d-2}$ we get
$$\frac{q_0}{(d-2)!}D^{d-2}_{d-2} = \frac{{\tilde e_0}}{(d-1)!}C^{d-1}_{d-1}m_0
 \implies
|{q_0}| \leq {\tilde e_0}\frac{C^{d-1}_{d-1}m_0}{|D^{d-2}_{d-2}|}
\leq {\tilde e_0}^2m_0.$$

Comparing the coefficients of $m^{d-i}$, where $2\leq i\leq d$,  we get
$$\frac{q_0}{(d-2)!}D^{d-2}_{d-i}-\frac{q_1}{(d-3)!}D^{d-3}_{d-i}+\cdots +
(-1)^{i-2}\frac{q_{i-2}}{(d-i)!}D^{d-i}_{d-i}$$
$$= \frac{{\tilde e_0}}{(d-1)!}\left(C^{d-1}_{d-1}m_0^{i-1}+ C^{d-1}_{d-2}m_0^{i-2}
+\cdots +C^{d-1}_{d-i+1}m_0\right)$$
$$ -\frac{{\tilde e_1}}{(d-2)!}\left(C^{d-2}_{d-2}m_0^{i-2}+ C^{d-2}_{d-3}m_0^{i-3}
+\cdots +C^{d-2}_{d-i+1}m_0\right) +
\cdots (-1)^{i-2}\frac{{\tilde e_{i-2}}}{(d+1-i)!}\left(C^{d-i+1}_{d+1-i}\right).$$

This implies that
$$|q_{i-2}||D_{d-i}^{d-i}|\leq |{\tilde e_0}|\left(C_{d-1}^{d-1}m_0^{i-1} +
\cdots + C^{d-1}_{d-i+1}m_0\right)  + |{\tilde e_1}|\left(C_{d-2}^{d-2}m_0^{i-2}+
\cdots +
C^{d-2}_{d-i+1}m_0\right)$$
$$+\cdots + |{\tilde e_{i-2}}|(C^{d+1-i}_{d+1-i}) +
\left(|q_0||D^{d-2}_{d-i}|+|q_1||D^{d-3}_{d-i}|+ \cdots +
|q_{i-3}||D^{d+1-i}_{d-i}|\right).$$

Now the proof follows by induction.

For ${\sQ}$ such that
$$0\longto \sO_X\longto \sO_X(m_0)\longto \sQ\longto 0,$$ we have
$\chi(X, \sQ(m-m_0)) = \chi(X, \sO_X(m))-\chi(X, \sO(m-m_0))$, so we get
same bound for $q_i's$ in terms of ${\tilde e_j}'s$ as above except
that now
$$D_j^n \in
\left\{\sum x_1^{i_1}\cdots x_n^{i_n}\mid {\sum_{l=1}^n i_l =  n-j},\quad
 0\leq j\leq n,\quad \{x_1,\ldots, x_n\} =
 \{1-m_0,\ldots, n-m_0\}\}\right\}.$$
Hence the lemma follows.
\end{proof}

\end{document}